\documentclass[oneside,reqno, twoside, 11pt]{amsart}  
\usepackage[pdftex]{graphicx}
\usepackage{lmodern}
\usepackage{comment}
\usepackage[english]{babel}   
\usepackage{upgreek}   

\usepackage{amsfonts}  
\usepackage{extpfeil}
\usepackage{verbatim}
\usepackage{BOONDOX-cal}

\usepackage{microtype}
\usepackage[new]{old-arrows}
\usepackage{amsmath}
\usepackage{tikz-cd}

\usepackage{amsthm}
\usepackage{mathtools}

\usepackage{bbm}
\usepackage{paralist}
\usepackage[T1]{fontenc}
\usepackage[colorlinks,citecolor=magenta,pagebackref=true,urlcolor=magenta, bookmarksdepth=3]{hyperref}
\usepackage[nodisplayskipstretch]{setspace}

\usepackage[all]{xy}

\usepackage{nicefrac}

\usepackage[font=small,labelfont=bf]{caption}
\usepackage[labelformat=simple]{subcaption}

\usepackage{epsfig}

\usepackage{mathrsfs}
\usepackage{appendix}  
\usepackage{color}
\usepackage[T1]{fontenc}  
\usepackage{amsmath, amssymb, amsthm, mathtools}
\usepackage[margin=1in]{geometry}
\usepackage{hyperref}
\usepackage[all]{xy}

\DeclareFontFamily{OT1}{pzc}{}
\DeclareFontShape{OT1}{pzc}{m}{it}{<-> s * [1.10] pzcmi7t}{}
\DeclareMathAlphabet{\mathpzc}{OT1}{pzc}{m}{it}

\makeatletter
\newtheorem*{rep@theorem}{\rep@title}
\newcommand{\newreptheorem}[2]{%
	\newenvironment{rep#1}[1]{%
		\def\rep@title{#2~\ref{##1}}%
		\begin{rep@theorem}}%
		{\end{rep@theorem}}}

\makeatother

\theoremstyle{plain}
\newtheorem*{thm*}{Theorem}
\newtheorem*{cor*}{Corollary}
\newtheorem{thm}{Theorem}[section]

\newtheorem*{lem*}{Lemma}

\newtheorem{prb}[thm]{Open Problem}

\theoremstyle{definition}

\newcommand{\supp}{\mathrm{supp}}

\newcommand{\Z}{\mathbb{Z}}
 \newcommand{\R}{\mathbb{R}}
 
\newcommand{\bbF}{\mathbb{F}}

\renewcommand{\AA}{\mathcal{A}}

\newcommand{\cB}{\mathcal{B}}

\newcommand{\cM}{\mathcal{M}}
\newcommand{\cF}{\mathcal{F}}
\newcommand{\cS}{\mathcal{S}}
\newcommand{\cO}{\mathcal{O}}
\newcommand{\cC}{\mathcal{C}}
\newcommand{\cG}{\mathcal{G}}
\newcommand{\cE}{\mathcal{E}}

\newcommand{\sbseq}{\subseteq}
\newcommand{\spseq}{\supseteq}

\newcommand{\w}{\mathbf{w}}

\newcommand{\vanish}[1]{}
\renewcommand{\a}{\mathbf{a}}

\def\coker{\mathrm{coker}}
\def\sbs\subset
\def\sbseq{\subseteq}

\def\langle{\left<}
\def\rangle{\right>}

\def\({\left(}
\def\){\right)}
\def\no={\,{\,|\!\!\!\!\!=\,\,}}

\def\no={\,{\,|\!\!\!\!\!=\,\,}}
\def\sbseq{\subseteq}

\def\sbseq{\subseteq}
\def\sbs\subset
\def\spseq{\supseteq}

\newcommand{\xqedhere}[2]{%
\rlap{\hbox to#1{\hfil\llap{\ensuremath{#2}}}}}

\newcommand{\cm}[1]{}

\newcommand\mr[1]{\mathrm{#1}}

\newcommand{\fld}{\mathbbm{k}}

\newcommand{\m}{\mathfrak{m}}

\newcommand\com[1]{}

\renewcommand\deg{\mr{deg}}

\newcommand{\h}{\mathbf{h}}
\newcommand\x{\mathbf{x}}
\newcommand{\y}{\mathbf{y}}
\newcommand{\z}{\mathbf{z}}

\renewcommand\t{\mathbf{t}}

\DeclareMathOperator\Tr{Tr}

\DeclareMathOperator{\Hom}{Hom}
\DeclareMathOperator{\Ext}{Ext}

\DeclareMathOperator{\Tot}{Tot}

\DeclareMathOperator{\Vol}{Vol}

\DeclareMathOperator{\Supp}{Supp}

\newcommand{\KK}{\mathcal{K}}

\theoremstyle{definition}
\newtheorem{theorem}{Theorem}[section]
\newtheorem{definition}[theorem]{Definition}
\newtheorem{remark}[theorem]{Remark}
\newtheorem{example}[theorem]{Example}

\newtheorem{lemma}[theorem]{Lemma}
\newtheorem{prop}[theorem]{Proposition}

\newtheorem{corollary}[theorem]{Corollary}

\title[Frobenius identities for the volume map on Cohen--Macaulay rings]{Frobenius identities for the volume map on Cohen--Macaulay rings}
\author[Karim Adiprasito]{Karim Alexander Adiprasito}
\address{{Karim Adiprasito \ \emph{and}\ Ryoshun Oba\ \emph{and}\ Vasiliki Petrotou}, Sorbonne Université and Université Paris Cité, CNRS, IMJ-PRG, F-75005 Paris, France}
\email{adiprasito@imj-prg.fr \emph{and} oba@imj-prg.fr \emph{and} petrotou@imj-prg.fr}
\author{Eric Katz}
\address{Eric Katz, Department of Mathematics, The Ohio State University, 231 West 18th Avenue,Columbus, OH 43210-1174}
\email{katz.60@osu.edu}

\author{Ryoshun Oba}

\author[Stavros Papadakis]{Stavros Argyrios Papadakis}
\address{{Stavros Papadakis}, Department of Mathematics, University of Ioannina, Ioannina, 45110, Greece}
\email{spapadak@uoi.gr}
\author[Vasiliki Petrotou]{Vasiliki Petrotou}

\keywords{}
\subjclass[2010]{}

\begin{document}

\begin{abstract}
We study the volume map on Artinian quotients of Cohen-Macaulay algebras in characteristic $p$, and the interaction between it and the action of Frobenius on resolutions. This allows us to provide a general, conceptual way to understand Parseval--Rayleigh identities, curious inhomogeneous identities on the volume map which were developed for the proof of the Ohsugi--Hibi conjecture in \cite{APPlattice}. This general perspective gives a new approach to generic Lefschetz theory. We use this perspective to do the following: we give sufficient conditions for anisotropy and the Hard Lefschetz property  for generic Artinian reductions of graded Gorenstein rings; we
study the codimension-$3$ Gorenstein quotient of a polynomial ring by the ideal generated by Pfaffians, proving a Parseval--Rayleigh identity and deriving anisotropy and Hard Lefschetz in characteristic $2$; we deduce the $g$-theorem for simplicial spheres and the Ohsugi--Hibi conjecture following previous work of Adiprasito, Papadakis, and Petrotou; and 
we provide further examples of Parseval--Rayleigh identities for Gorenstein rings.
\end{abstract}

\maketitle

\section{Introduction}

In \cite{APPlattice}, Adiprasito, Papadakis and Petrotou, consider lattice spheres $\Sigma$ of dimension $(d-1)$, they consider the toric face ring over $\Sigma$ over a field of characteristic two, and describe a choice of isomorphism $\Vol\colon A^d(\Sigma)\to \fld$.
They prove the following inhomogeneous identity, which they called the \emph{Parseval--Rayleigh identity}
\[\Vol\left(\x^{\mathbf{g}}\right)=(-1)^{m}\sum_{B\in\cB} \Vol\left(\x^{\frac{\mathbf{g}+\overline{B}}{p}}\right)^p \frac{\theta^B}{B!}.\]
Using this, they conclude the Hard Lefschetz property for lattice spheres in characteristic two.

The purpose of this paper is to systematize this work and relate it to the action of Frobenius, which makes the Parseval--Rayleigh identity not just immediate but inevitable, and describe Parseval--Rayleigh identities for general Cohen-Macaulay rings. We refine the homological approach to the volume map introduced in \cite{APPKM}:
\[\Vol\colon (\omega_A)_0\to \fld\]
by recognizing it as the inverse of a trace map, see Definition~\ref{def:volume-map}.

From this definition, we are able to rephrase the Parseval--Rayleigh identity as an immediate consequence of the Frobenius-equivariance of the volume map. This yields the \emph{abstract Parseval--Rayleigh identity}
\[\Vol(\z)=\Vol(\Phi^*_A(\z))^p\]
on Artinian quotients $R/I$ of Cohen--Macaulay graded standard $\fld$-algebras $R$ where $\Phi^*_A\colon \varphi_*\omega_A\to \omega_A$ is the Frobenius trace, see Theorem~\ref{t:abstractpr}. In the case of Gorenstein rings, this was discovered independently by Mykola Pochekai \cite[Theorem~1.1]{Po}; his proof is similar to ours, and both take inspiration from the original arguments in \cite{APPlattice, AOPP:complete}, perhaps underlining how inevitable this form of the Parseval--Rayleigh identity is. Still, it is the action of Frobenius that is the new insight in both that conceptualizes the proof in a new way.

In characteristic $2$, we recover the Parseval--Rayleigh identity of \cite{APPlattice}, one of the main ingredients of the proof of the Ohsugi--Hibi conjecture, as it implies an anisotropy property of toric face rings. This identity encompasses the differential identities for the volume polynomial in positive characteristic \cite{PP, APPhal}, shedding more light onto the second proof of the $g$-conjecture in characteristic two. In the case of face rings, Parseval--Rayleigh identities were also announced in \cite{KLS}. We hope that this work clarifies these mysterious identities. 

In Section~\ref{s:parsevalimpliesanisotropy}, we give a criterion after \cite{AOPP:complete} and \cite {KLS} to derive anisotropy and the Lefschetz property from the Parseval--Rayleigh identity. We recover the Lefschetz property for IDP lattice polytopes, simplicial spheres and complete intersections, deducing the $g$-theorem for simplicial spheres in characteristic $2$ and the Ohsugi--Hibi conjecture in Section~\ref{s:latticepolytopes}, following \cite{APPlattice}.

Our main new application is to the codimension-$3$ Lefschetz conjecture of Migliore and Zanello \cite[Question 3.8]{MZ}.

\begin{prb}
	Do all codimension-$3$ Gorenstein algebras in characteristic 0 possess the weak Lefschetz property?
\end{prb}

We address this question in characteristic two, see Section~\ref{s:pfaffians}, and use the Parseval--Rayleigh identity to prove the anisotropy for the generic Pfaffian in characteristic two (which, in turn, implies that for characteristic zero). This gives a new proof for the generic Lefschetz property, which was previously shown using different methods \cite{ABD}. 


We also provide some additional examples of Parseval--Rayleigh identities on Gorenstein rings in Section~\ref{s:additional}, hinting at future developments.

There are many open conjectures about the Lefschetz properties of generic Artinian reductions on Cohen--Macaulay and Gorenstein rings.
For example, Migliore and Nagel ask if the generic Artinian reductions of reduced standard $\fld$-algebras have the weak Lefschetz property \cite[Question~3.8]{MN:tour}).  \cite{Braun:unimodality} asked whether generic Artinian reductions of Gorenstein integral domains have the Lefschetz property. 
We also mention that some non-generic reductions are relevant to geometry. For example, the Jacobian ring of a hypersurface with given Newton polytope \cite{Batyrev:variations} is an Artinian quotient of a semigroup ring to which our methods do not directly apply. 
While we cannot yet answer such questions, we believe our methods and the Parseval--Rayleigh identity in other settings might give an approach to such conjectures. 

We would like to acknowledge Patrick Brosnan, Mircea Mustata, Frank-Olaf Schreyer, and Karen Smith for helpful discussions. We benefitted greatly from using the
computer algebra program Macaulay2 \cite{M2} and thank its developers and contributors. K.\,A.\,A., R.\,O., and V.\,P., were supported by Horizon Europe ERC Grant
number: 101045750 / Project
acronym: HodgeGeoComb. R. O. was supported by JSPS Overseas Research
Fellowships.

\section{The Volume Map}

We define the volume map on an Artinian ring $A\coloneqq R/I$ where $R$ is a $d$-dimensional finitely generated Cohen--Macaulay graded standard $\fld$-algebra using some ideas from \cite{APPKM}.

Let $\m$ be the maximal ideal of $R$ generated by positive degree elements. 
We have the identification of the canonical modules of $R/\m\cong \fld$ and $R/I$ by \cite[Theorem~3.3.7]{BrunsHerzog},
\[
  \omega_{R/\m} \cong \operatorname{Ext}^{d}_R(R/\m,\omega_R),
  \qquad
  \omega_{R/I} \cong \operatorname{Ext}^{d}_R(R/I,\omega_R).
\]
These isomorphisms can be interpreted as graded local duality \cite[Theorem 3.6.19(b)]{BrunsHerzog}.

Since $I \subset \m$, there exists a natural surjective map
\[
  A\coloneqq R/I \longrightarrow R/\m\cong \fld,
\]
which we claim induces an injective trace map
\[
  \Tr_{k/A} : \omega_{R/\m} \cong \operatorname{Ext}^d_R(R/\m,\omega_R)
  \to
  \omega_{A} \cong \operatorname{Ext}^d_R(A,\omega_R).
\]
Indeed, from the exactness of 
\[\
\xymatrix{
0\ar[r]&\m A\ar[r]&A\ar[r]&\fld\ar[r] &0
},\]
we have the exactness of
\[\xymatrix{
\Ext_R^{d-1}(\m A,\omega_R)\ar[r]&\Ext_R^d(\fld,\omega_R) \ar[r]&\Ext_R^d(A,\omega_R)
}.\]
The leftmost term vanishes because $\Ext_R^{d-1}(\m A,\omega_R)\cong H^1_{\m}(\m A)^\vee$ and $\m A$ is a finite length $R$-module.
Now, $\omega_{\fld}\cong \fld$, so we can fix a graded $R$--module isomorphism between them. Since $A$ is generated in positive degree, $(\omega_A)_0\cong \fld$. 

\begin{definition}\label{def:volume-map}
The \emph{volume map} on $A$ is the inverse to $\Tr_{\fld/A}$ on the zeroth graded part,
$\Vol\colon (\omega_A)_0 \to \omega_\fld\cong \fld$.
\end{definition}

\begin{remark}
  The volume map can be computed from a minimal free resolution $F_\bullet\to R/I$ over $R$. Now, $k\cong R/\m$ is resolved over $R$  by some complex $G_\bullet$. The quotient homomorphism $R/I\to R/\m$ extends to a chain map (unique up to chain homotopy), $F_\bullet\to G_\bullet$ to which we can apply $\Hom(\cdot,\omega_R)$ to obtain $\Hom(G_\bullet,\omega_R)\to \Hom(F_\bullet,\omega_R)$. This induces an injection $H^d(\Hom(G_\bullet,\omega_R))\to H^d(\Hom(F_\bullet,\omega_R))$ which is an isomorphism on the $0$th graded component of the target.
When $R\cong \fld[x_1,\dots,x_d]$, we can let $G_\bullet\coloneqq \KK((x_1,\dots,x_d),R)$ be the Koszul complex for $R/\m$.
\end{remark}


The volume map has a naturality property for quotients. If $I_1\subset I_2$ are ideals in $R$ inducing a quotient map of Artinian rings $A_1\coloneqq R/I_1\to A_2\coloneqq R/I_2$, then for $z\in(\omega_{A_2})_0$, we define $\Tr_{A_2/A_1}(z)\in (\omega_{A_1})_0$ to be the image of $z$ under $\Ext^d(R/I_2,\omega_R)\to \Ext^d(R/I_1,\omega_R)$. By considering the composition, $R/I_1\to R/I_2\to R/\m$, we immediately obtain
$\Vol_{A_1}(\Tr_{A_2/A_1}(z))=\Vol_{A_2}(z).$


\section{The general Parseval--Rayleigh Identity}

\subsection{The action of Frobenius}
In this section, we will suppose that $R'$ is a finitely generated graded $d$-dimensional Cohen--Macaulay $\bbF_p$-algebra. Write $R\coloneqq R'\otimes_{\bbF_p} \fld$ and $\overline{R}\coloneqq R'\otimes_{\bbF_p}\overline{\fld}$ where $\fld$ is a field extension of $\bbF_p$ and $\overline{\fld}$ is its algebraic closure. Write $\varphi\colon R'\to R'$ for the usual $p$-power Frobenius.
Recall that for an $R'$-module $M$, we write $\varphi_*M$ for the module $M$ equipped with the $R'$-action $r\cdot m=r^pm$. This is an exact functor. There is a Frobenius trace on $\Phi_{R'}^*\colon \varphi_*\omega_{R'}\to \omega_{R'}$ dual to $\varphi$, defined by writing $\omega_{R'}\cong H^d_{\m}(R')^\vee$ and having Frobenius act on the \v{C}ech resolution used to compute $H^*_{\m}(R')$. Similarly, there is a Frobenius trace $\Phi_{\overline{R}}\colon \varphi_*\omega_{\overline{R}}\to \omega_{\overline{R}}$ which satisfies $\Phi_{\overline{R}}=\varphi_{\overline{k}}^{-1}\circ (\Phi_R\otimes \overline{\fld})$ where $\varphi_{\overline{\fld}}$ is the Frobenius on $\overline{\fld}$.

\begin{example} \label{e:polyring} We can write the Frobenius trace on $\overline{R}\coloneqq \overline{\fld}[x_1,\dots,x_d]$. 
The $p$-power map on the \v{C}ech complex, $C^\bullet\to \varphi_*C^\bullet$  induces a Frobenius action on local cohomology module $H^d_{\m}(\overline{R})$ and by duality, the Frobenius trace on 
\[\omega_{\overline{R}}\cong  (x_1\dots x_d)\overline{\fld}[x_1,\dots,x_d]\]
 given by 
 \[\x^{\a}\mapsto 
\begin{cases}
 \x^{\a/p}&\text{if } p\mid\a\\
 0&\text{else}
 \end{cases}\]
 where $p\mid\a$ means that $p$ divides every component of $\a$.
Under the isomorphism $\omega_{\overline{R}}\cong \overline{\fld}[x_1,\dots,x_d](-d)$, it has the more familiar form
 \[\x^{\a}\mapsto 
\begin{cases}
 \x^{\frac{\a+\mathbf{1}}{p}-\mathbf{1}}&\text{if } p\mid(\a+\mathbf{1})\\
 0&\text{else}
 \end{cases}\]
 where $\mathbf{1}$ denotes the vector of all ones. 
\end{example}

We will define a Frobenius trace on a codimension-$t$ Cohen--Macaulay ring $A\coloneqq R/I$ by taking a minimal free resolution $F_\bullet\to A$ over $R$ and observing 
 \[\omega_A\cong \Ext^t(A,\omega_R)\cong H^t(\Hom(F_\bullet,\omega_R)).\]
 For convenience, we will write a chain map $F_\bullet\to \varphi_*F_\bullet$ lifting $\varphi_*$ by applying adjunction to a chain map $\Phi_\bullet\colon \varphi^*F_\bullet\to F_\bullet$, following \cite[Section~8.2]{BrunsHerzog}. Let $R^{\varphi}$ be the ($R$-$R$)-bimodule with underlying group $R$ equipped with actions
\[a\cdot r \circ b=arb^p,\quad a,b\in R,\ r\in R^{\varphi}.\]
For an $R$-module $F$, write $\varphi^*F\coloneqq R^{\varphi}\otimes F$.
The adjoint of the Frobenius map $\varphi\colon A\to \varphi_*A$ is a map $\varphi^*A\to A$ that can be identified as the natural surjection $R/I^{[p]}\to R/I$ where $I^{[p]}$ is the ideal obtained by applying the $p$-power map to each element of $I$.
 For a complex $(F_\bullet,\partial_\bullet)$, $\varphi_*F_\bullet$ is 
\[\xymatrix{
\cdots \ar[r]& \varphi^*F_d \ar[r]^{\partial_d} & \varphi^*F_{d-1}  \ar[r]^>>>>{\partial_{d-1}}&\dots\ar[r]^{\partial_{2}}&\varphi^*F_1 \ar[r]^{\partial_1}&\varphi^*F_0.
}\]
Explicitly, if we fix an isomorphism $F_i\cong R^{n_i}$, it can be written as
\[\xymatrix{
\ \ar[r]& R^{n_d} \ar[r]^{\varphi(M_d)} & R^{n_{d-1}}\ar[r]^>>>>{\varphi(M_{d-1})}&\dots\ar[r]^{\varphi(M_2)}&R^{n_1}\ar[r]^{\varphi(M_1)}&R^{n_0}.
}\]
where $M_i$ is the matrix representing $\partial_i$ and $\varphi(M_i)$ 
denotes the Frobenius applied entrywise to $M_i$. Now, because $\varphi^*F_\bullet$ is a complex of free $R$-modules and $F_\bullet$ is exact in positive degrees, there is a chain map $\Phi^{\flat}_\bullet\colon \varphi^*F_\bullet\to F_\bullet$ such that $\Phi^{\flat}_0\colon \varphi^*F_0\to F_0$ is the identity. By adjunction, $\Phi^{\flat}_\bullet\colon \varphi^*F_\bullet\to F_\bullet$ induces a chain map $\Phi_\bullet\colon F_\bullet\to \varphi_*F_\bullet$.

\begin{definition} 
Let $F_\bullet\to A$ be a minimal free resolution of $A$ over $R$. Let $\Phi_{\bullet}\colon F_\bullet\to \varphi_*F_\bullet$ be a chain map lifting the $p$-power map on $A$. Define $\Phi^*_{A,i}$ by
\[\Phi^*_{A,i}\colon \varphi_*\Hom(F_i,\omega_R)\to \Hom(F_i,\omega_R)\otimes\overline{\fld},\quad \psi\mapsto \Phi_{\overline{R}}^*\circ\psi\circ \Phi_i.\]
The {\em Frobenius trace} $\Phi^*_A\colon \varphi_*\omega_A\to\omega_A\otimes\overline{\fld}$ is defined to be the map induced by $\Phi^*_{A,t}$:
\[\varphi_*\omega_A\cong \varphi_*\Ext^t(A,\omega_R)\cong \varphi_* H^t(\Hom(F_i,\omega_R))\to H^t(\Hom(F_i,\omega_R))\otimes\overline{\fld}\cong \Ext^t(A,\omega_R)\otimes\overline{\fld}\cong \omega_A\otimes\overline{\fld}.\]
\end{definition}

For $\fld$ perfect, this definition agrees with the usual Frobenius trace on $A$. Indeed, the isomorphism $\omega_A\cong \Ext^t(A,\omega_R)$ arises from considering the double complex $E_0^{i,-j}\coloneqq \Hom_R(F_i,(C^j)^\vee)$
where $(C^\bullet)^\vee$ is the graded dual of the \v{C}ech resolution of $R$. We compute the cohomology of the total complex $T^\bullet$ (where $T^m\coloneqq\bigoplus_j E_0^{m+j,-j}$) by a spectral sequence of the double complex in two different ways. Because the formation of the \v{C}ech complex arises from taking the direct sum of localizations (which is an exact functor), one spectral sequence has the following pages: $E_1^{i,-j}$ is $\Hom_R(A,(C_j)^\vee)$ at $i=0$ and zero elsewhere; $E_2^{i,-j}$ is nonzero only at $(i,j)=(0,d-t)$ where it is $\omega_A$ . The other spectral sequence has the following pages: $E_1^{i,-j}$ is $\Hom_R(F_i,\omega_R)$ at $j=d$ and zero elsewhere; $E_2^{i,-j}$ is nonzero only at $(t,d)$ where it is $\Ext^t_R(A,\omega_R)$. Hence, one obtains the isomorphism $\omega_A\cong \Ext^t(A,\omega_R)$.  Now, the map $\psi\mapsto \Phi_R^*\circ\psi\circ \Phi_i$ gives a map of double complexes $\varphi_*E_0^{i,-j}\to E_0^{i,-j}$ which induces the usual Frobenius trace on $\omega_A$ and the map defined above on $\Ext^t_R(A,\omega_R)$.

Henceforth, we will suppose that the isomorphism $\psi\colon \omega_{\fld}\otimes\overline{\fld} \to \overline{\fld}$ is chosen to be Frobenius-equivariant, i.e., for $z\in (\omega_{\overline{k}})_0$,
$\psi(\Phi_{\overline{k}}(z))=\psi(z)^{1/p}.$


\begin{theorem}[Abstract Parseval--Rayleigh Identity] \label{t:abstractpr} Let $A\coloneqq R/I$ be an Artinian quotient of a $d$-dimensional standard Cohen--Macaulay $\fld$-algebra.
 For $\z\in (\omega_A)_0$, we have
 $\Vol(\z)=\Vol(\Phi^*_A(\z))^p.$
 \end{theorem}
  
\begin{proof}
Let $\overline{R}\coloneqq R\otimes\overline{\fld}$ and $\overline{A}\coloneqq \overline{R}/\overline{I}$ where $\overline{I}\subset \overline{R}$ is the ideal generated by $I$.
We must show the following diagram commutes:
\[\xymatrix{
\varphi_*\omega_{\overline{\fld}}\ar[d]_{\Phi^*_{\overline{\fld}}}\ar[r]^{\varphi_*\Tr_{{\overline{k}}/\overline{A}}}&\varphi_*\omega_{\overline{A}}\ar[d]^{\Phi^*_{\overline{A}}}\\
\omega_{\overline{\fld}}\ar[r]_{\Tr_{{\overline{\fld}}/{\overline{A}}}}&\omega_{\overline{A}}
}\]
which demonstrates that $\Vol(\Phi_{\overline{A}}^*(\z))=\Vol(\z)^{1/p}$. 

Find minimal free resolutions $F_\bullet\to \overline{A}$ and $G_\bullet\to \overline{\fld}$ over $\overline{R}$. Construct chain maps $F_\bullet\to G_\bullet$, $F_\bullet\to\varphi_*F_\bullet$, and $G_\bullet\to\varphi_*G_\bullet$ lifting the map $\overline{A}\to \overline{\fld}$, and the Frobenius maps on $\overline{A}$ and $\overline{\fld}$, respectively. These maps together with the restriction of scalars $\varphi_*F_\bullet\to \varphi_*G_\bullet$ induce $\Tr_{\overline{\fld}/\overline{A}}, \varphi_*\Tr_{\overline{\fld}/\overline{A}}, \Phi^*_{\overline{A}},\text{ and }\Phi^*_{\overline{\fld}}$.
Now, we need only that the following diagram of complexes commutes up to chain homotopy
\[\xymatrix{
\varphi_*G_\bullet&\ar[l]\varphi_*F_\bullet\\
G_\bullet\ar[u]&\ar[l]F_\bullet.\ar[u]
}\]
This follows because the two compositions $F_\bullet\to\varphi_*G_\bullet$ map a free complex to an exact complex and agree for $\bullet=0$.
\end{proof}

\begin{remark}
The map $\Phi_\bullet$ is only unique up to chain homotopy. Thus, the behavior of $\Phi_A^*$ on $\omega_A$ is unique while its expression as a function on $\Hom(F_d,\omega_R)$ is not.
\end{remark}

\begin{example} \label{e:parsevalinpolynomialrings}
We can specialize the Parseval--Rayleigh identity to the case where $R\cong \fld[x_1,\dots,x_d]$. Then $\Phi^{\flat}_\bullet$ is a chain map 
\[\xymatrix{
0\ar[r] &F_d \ar[r]^{\varphi(M_d)}\ar[d]^{\Phi^{\flat}_d}&F_{d-1}\ar[r]^{\varphi(M_{d-1})}\ar[r]\ar[d]^{\Phi^{\flat}_{d-1}}&\cdots\ar[r]^{\varphi(M_1)}&F_0\ar[r]\ar[d]^{\Phi^{\flat}_0}&0\\
0\ar[r] &F_d \ar[r]^{M_d}&F_{d-1}\ar[r]^{M_{d-1}}\ar[r]&\cdots\ar[r]^{M_1}&F_0\ar[r]&0
}\]
where $\Phi^{\flat}_0$ is the identity. Now, 
\[\omega_R\cong R(-d),\quad \omega_A\cong \coker(F_{d-1}^\vee\to F_d^\vee)(-d).\]
Pick a basis $\{e_\alpha\}$ for $F_d$ and thus a dual basis $\{e_\alpha^\vee\}$, so we can define a  pairing between $F_d^\vee$ and $F_d$  by the Kronecker pairing
$\langle \x^\a e_{\alpha}^\vee,\x^{\mathbf{b}} e_{\beta}\rangle=\delta_{\a\mathbf{b}}\delta_{\alpha\beta}$.
For $\z\in \Hom(F_d,R)(-d)$, we have 
\[\langle \Phi_R^*(\z),\x^{\a}e_\alpha \rangle=\langle \z, (x_1\dots x_d)^{p-1} \x^{p\a}e_\alpha\rangle^{1/p}.\]

Then, we can express $\Phi_A^*\colon \varphi_*\omega_A\to \omega_A$ by for $\z\in \Hom(F_d,R)(-d)$ and a monomial $\x^{\a}\in \fld[x_1,\dots,x_d]$,
\[\langle \Phi_A^*(\z),\x^{\a} e_\alpha\rangle=\langle \z,(x_1\dots x_d)^{p-1}\x^{p\a}\Phi^{\flat}_d(e_{\alpha})\rangle^{1/p}.\]
Consequently, the Parseval--Rayleigh identity takes the form
\[\Vol(\z)=\sum_{\alpha,\a}\langle \Phi_d^*\z,(x_1\dots x_d)^{p-1}\x^{p\a}e_{\alpha}\rangle\Vol\left(\x^{\a}e_{\alpha}^\vee\right)^p\]
where $\Phi_d^*$ denotes the dual to $\Phi_d$.
\end{example}

\subsection{The Parseval--Rayleigh identity for Artinian reductions of Cohen--Macaulay quotients}

Given a codimension-$u$ Cohen--Macaulay ring $A\coloneqq R/I$ for $R\coloneqq \fld[x_1,\dots,x_n]$ with a minimal free resolution $F_\bullet\to A$ of length $u$, and homogeneous $\ell\in R$ descending to a nonzero divisor on $A$, we can write a free resolution of $A/(\ell)$ over $R$ and thus a Parseval--Rayleigh identity. Indeed, first consider the Koszul resolution of $R/(\ell)$,  $\KK_\bullet\coloneqq [R\xrightarrow{\ell\cdot} R]$ which has the following expression for $\Phi_{\KK_\bullet}^{\flat}\colon \varphi^*\KK_\bullet\to \KK_\bullet$:
\[\xymatrix{
R\ar[r]^{\ell^p\cdot}\ar[d]_{\ell^{p-1}\cdot}&R\ar[d]^{\operatorname{Id}}\\
R\ar[r]_{\ell\cdot}&R.
}\]
Consequently, $\Phi^\flat_{\KK_\bullet,1}\colon R\to R$ is multiplication by $\ell^{p-1}$. The total complex attached to the double complex $F_\bullet\otimes \KK_\bullet$ gives a free resolution of $A/(\ell)$ such that 
\[\Phi^{\flat}_{u+1}\colon \varphi^*(\Tot(F_\bullet\otimes \KK_\bullet)_{u+1})\cong \varphi^*(F_u\otimes \KK_1)\cong \varphi^*F_u\to F_u\cong F_u\otimes \KK_1\cong \Tot(F_\bullet\otimes \KK_\bullet)_{u+1}\]
is given by $\Phi^{\flat}_{F_\bullet,u}\otimes \ell^{p-1}$.
Similarly, by induction, if $(\ell_1,\dots,\ell_t)$ descends to a regular sequence in $A$, we can tensor together $t$ copies of the above complex to produce the Koszul complex $\KK_\bullet\to R/(\ell_1,\dots,\ell_t)$. Then, $\Phi^{\flat}_{u+t}\colon \varphi^*(\Tot(F_\bullet\otimes \KK_\bullet)_{u+t})\to \Tot(F_\bullet\otimes \KK_\bullet)_{u+t}$ is given by $\Phi^{\flat}_{F_\bullet,u}\otimes \ell_1^{p-1}\dots\ell_t^{p-1}$.

If $u+t=\dim R$, then $\AA\coloneqq A/\boldsymbol{\ell}\coloneqq A/(\ell_1,\dots,\ell_t)$ is Artinian with top degree $\max_j(a_j)+ \sum_i \deg \ell_i- n$ where $F_u = \bigoplus_j R(-a_j)$. The Parseval--Rayleigh identity becomes for $\z\in (\omega_{\AA})_0$,
\[\Vol(\z)=\Vol(\Phi_R(\ell_1^{p-1}\dots\ell_t^{p-1}\Phi_{F_\bullet,u}^*(\z)))^p.\]

Now, we can write the Parseval--Rayleigh identity for homogeneous complete intersections by using a regular system of parameters of degrees $(m_1,\dots,m_t)$ whose coefficients are indeterminates. Let $S_i\subset R^{m_i}$ be a finite set for $i=1,\dots,t$. We introduce the following field and homogeneous system:
\[\fld\coloneqq \bbF_p\left(\{\theta_{i,\y}\mid 1\leq i\leq t,\ \y\in S_i\}\right),\quad
\ell_i\coloneqq \sum_{\y\in S_i}\theta_{i,\y}\y. \]
We suppose that $(\ell_1,\dots,\ell_t)$ is a homogeneous regular system of parameters.
 

We will write the Parseval--Rayleigh identity using the notation of \cite{KLS}.
Consider functions $B\colon S_1\sqcup\dots\sqcup S_t\to \Z_{\geq 0}$
to index monomials in the expansion of $\ell_1^{p-1}\dots \ell_t^{p-1}$.
Write
\[\theta^B=\prod_{i}\prod_{\y\in S_i} \theta_{i,\y}^{B(\y)},\ B!=\prod_{i}\prod_{\y\in S_i} B(\y)!.\]
Let $\cB$ be the set of functions $B$ whose values on each $S_i$ sum to $p-1$, i.e.,
\[\cB\coloneqq \left\{B\colon S_1\sqcup\dots \sqcup S_t\to \Z_{\geq 0}\ {\Bigg|}\ \sum_{\y\in S_i} B(\y)=p-1\text{ for all i}\right\}.\]
Write $\y^\cB\coloneqq \prod_i \prod_{\y\in S_i} \y^{B(\y)}$.

\begin{prop} \label{p:parsevalartinian} For any $\z\in (\omega_{A/(\boldsymbol{\ell})})_0$,
\[\Vol\left(\z\right)=(-1)^{t}\sum_{B\in\cB} \Vol\left(\Phi_R(\y^B\Phi^*_{F_\bullet,u}(\z))\right)^p \frac{\theta^B}{B!}.\]
\end{prop}

\begin{proof}
Observe
\[\ell_i^{p-1}=\sum_{\mathbf{b}}
(p-1)!\prod_{\mathbf{y}\in S_i}
 \frac{(\theta_{i,\y}\y)^{\mathbf{b}(\y)}}{\mathbf{b}(\y)!}
=-\sum_{\mathbf{b}}
\prod_{\mathbf{y}\in S_i}
 \frac{(\theta_{i,\y}\y)^{\mathbf{b}(\y)}}{\mathbf{b}(\y)!}
\]
where the sum is over $\mathbf{b}\colon  S_i\to \Z_{\geq 0}$ with $\sum_{\y\in S_i} b(\y)=p-1$, and
 we used Wilson's theorem to write $(p-1)!=-1$.
Consequently, 
\[\ell_1^{p-1}\dots \ell_t^{p-1}=(-1)^t\sum_{B\in\cB}\y^{\overline{\mathbf{B}}}\frac{\theta^B}{B!},\]
and thus,
\begin{multline*}
\Vol\left(\z\right)=\Vol\left(\Phi_{A/\boldsymbol{\ell}}^*(\z)\right)^p=\Vol\left(\Phi_R^*(\ell_1^{p-1}\dots \ell_t^{p-1}\Phi^*_{F_\bullet,u}(\z))\right)^p\\
=(-1)^{t}\sum_{B\in\cB} \Vol\left(\Phi_R^*\left(\y^B\Phi^*_{F_\bullet,u}(\z)\right)\left(\frac{\theta^B}{B!}\right)^{1/p}\right)^p.
\qedhere
\end{multline*}

\end{proof}

%

The following is immediate after recognizing that $p$-powers are constants over differential fields of characteristic $p$:

\begin{corollary} \label{c:differentialidentity} Write $\partial_{\theta_i,\y}$ for the derivative with respect to $\theta_{i,\y}$ for $\y\in S_i$. For $B\in\cB$, set $\partial^B\coloneqq\prod_{i}\prod_{\y\in S_i} \partial_{\theta_{i,\y}}^{B(\y)}.$
Then,
$\partial^B\Vol\left(\z\right)=(-1)^{t}\Vol\left(\Phi^*_R(\y^B\Phi^*_{F_\bullet,u}(\z))\right)^p.$
\end{corollary}

When $I=(0)$, $R=A$, $u=0$, and then $\Phi_{F_\bullet,0}=\operatorname{Id}$.

\subsection{A Cohen--Macaulay example}

In this section, we consider the Parseval--Rayleigh identity for a non-Gorenstein ring.
Consider the  standard graded polynomial 
ring  $R'\coloneqq \bbF_p[x_1, x_3, x_3]$ and the codimension 2 homogeneous  Cohen--Macaulay ideal 
\[
    I'  \coloneqq  (x_2x_3,x_1x_3,x_1x_2)  \;  \subset R'.
\]
Let $\fld\coloneqq \bbF_p(a,b,c)$ and set $R\coloneqq \fld[x_1,x_2,x_3]$ and 
$\ell\coloneqq ax_1+bx_2+cx_3$.
Set $A'\coloneqq R'/I'$ and $A\coloneqq R/I$ where $I$ is the ideal of $R$ generated by $I'$ under the inclusion $R'\hookrightarrow R$. The generic Artinian reduction of $A$ is 
$\AA\coloneqq A/(\ell)$.
The Hilbert function of $\AA$ is 
\[
    0 \mapsto 1,  \;  \;  \;   1 \mapsto 2,  \;   \;  \;   i \mapsto 0  \; \;   \;   \text { otherwise.} 
\]
Hence, the top degree $s$ of   $\AA$ is equal to $1$.  
%
By studying the minimal free resolution, we see that $\omega_{\AA}$ is generated, as an $R$-module,  by two elements 
$e_1,e_2$ of degree $-1$, with relations
\[
         x_3  e_1 = 0,  \quad    -x_2 e_1 + x_2 e_2 = 0,    \quad 
            x_1 e_2 = 0,   \quad     ax_1e_1 +b x_2e_1 = 0,  \quad  b x_2  e_2 + c x_3  e_2 = 0.
\]
We fix the volume map by setting  $\Vol(z_0)=1$ where $z_0\coloneqq c x_3 e_2 \in \omega_{\AA}$, hence 
   \[  \begin{matrix}
     \Vol (x_1e_1) = 1/a,   &      \Vol (x_2e_1) = -1/b,   &
                              \Vol (x_3e_1) = 0,   \\
     \Vol (x_1e_2) =  0,   &      \Vol (x_2e_2) = -1/b,   &
                              \Vol (x_3e_2) = 1/c.
       \end{matrix}                       
 \]
     
 By the above
 computations, we get the following identities of the same shape as the Parseval--Rayleigh identities:
 \begin{alignat*}{4}
           \Vol (x_1e_1)& =  a^{p-1} (\Vol (x_1e_1))^p,   && 
           \Vol (x_2e_1) &&=  (-b)^{p-1} (\Vol (x_2e_1))^p,   \\           
           \Vol (x_3e_1) &=  c^{p-1}(\Vol (x_3e_1))^p,   && 
           \Vol (x_1e_2) &&=   a^{p-1} (\Vol (x_1e_2))^p,       \\       
           \Vol (x_2e_2) &=   (-b)^{p-1}(\Vol (x_2e_2))^p,     \quad&&           \Vol (x_3e_2) &&=   c^{p-1}  (\Vol (x_3e_2))^p.
 \end{alignat*}
 
 We consider the minimal graded free resolution 
     \[
             F_\bullet:  \quad  0  \to R(-4)^2\to R(-3)^{5} 
                            \to  R(-2)^3\oplus R(-1)  \to  R
     \]
           of  $A$   as  $R$-module, and the $p$-th Frobenius power 
     \[
             \varphi^*F_\bullet:  \quad  0  \to R(-4p)^2\to R(-3p)^{5} 
                            \to  R(-2p)^3\oplus R(-p)  \to  R
     \]

    Then there exists 
           a chain map $\Phi^\flat_\bullet\colon \varphi^*F_\bullet\to F_\bullet$ 
    such that $\Phi^\flat_0$ is the identity map and $\Phi^{\flat}_3$
    can be represented by the matrix 
 \[
     G = \begin{pmatrix}  (x_1x_2x_3)^{p-1} (ax_1+bx_2+cx_3)^{p-1} & 0  \\ 
                   0 & (x_1x_2x_3)^{p-1} (ax_1+bx_2+cx_3)^{p-1}  \end{pmatrix}.
  \]

Hence, the following formula given by Theorem~\ref{t:abstractpr} holds
for all  $i \in \{1,2 \}$ and   $w \in \{x_1,x_2,x_3\}$
\[
    \Vol( we_i ) = \sum_{j=1}^2 \sum_{u=1}^3
      \langle Gwe_i,(x_1x_2x_3)^{p-1}x_u^p e_j\rangle (\Vol(x_ue_j))^p 
\]
which immediately implies the above Parseval--Rayleigh-like identities.

\begin{remark} Related Macaulay2 code is contained in the file  \cite{AKOPP51}.
\end{remark}

\section{The Parseval--Rayleigh identity, Anisotropy, and the Lefschetz property} \label{s:parsevalimpliesanisotropy}

We now extract some minimal conditions needed to establish anisotropy from the Parseval--Rayleigh identities. Our technique follows an argument that appears in \cite{KLS} and \cite{AOPP:complete}.

\subsection{Quotients of Polynomial Rings} \label{ss:quotientsofpolynomialrings}

We will give conditions for anisotropy and Hard Lefschetz on the Artinian reduction of Gorenstein quotients of polynomial rings.

Write $R'=\bbF_p[x_1,\dots,x_n]$, and let $A'\coloneqq R'/I'$ be a $d$-dimensional  Gorenstein quotient, so there is a minimal free resolution $F_\bullet\to A'$ over $R'$ of length $u\coloneqq n-d$ such that $F_{n-d}\cong R'(-c)$ for some $c\in \Z_{\geq 0}$. Construct a chain map $\Phi_\bullet^\flat\colon\varphi^*F_\bullet\to F_\bullet$, and write $H_0\colon \varphi^*F_{n-d}\to F_{n-d}$, considered as an element $H_0\in (R')^{(p-1)c}$. We call $H_0$ a {\em Parseval core}. We consider the support $\supp(H_0)$ as a subset of $\Z_{\geq 0}^n$. Set $s\coloneqq c+d-n$.

For $S\subseteq \Z_{\geq 0}^n$ and an integer $q$, we define 
$S_{\geq q}\coloneqq \{(m_1,\dots m_n)\in S\mid m_1+\dots+m_n\geq q\}.$
\begin{definition}
 The Parseval core $H_0$ is {\em $p$-conducting} if 
for any $\mathbf{a} \in \Z_{\geq 0}^n$ of total degree at least $(p-1)s$, there exists $\mathbf{m} \in \supp(H_0)$ such that:
\begin{enumerate}
    \item \label{i:smallermonomial} $\mathbf{a} + (p-1)\mathbf{1} - \mathbf{m} \in \Z_{\geq 0}^n$, and
    \item \label{i:paritynonoverlapping} if $\mathbf{m}' \in \supp(H_0)$ satisfies $\mathbf{m} - \mathbf{m}' \in p\Z^n$, then $\mathbf{m} = \mathbf{m}'$.
\end{enumerate}
\end{definition}

\begin{remark}
  It has been our practice to pick representatives of 
  \[\omega_{A'}\cong \coker(\Hom(F_{u-1},\omega_{R'})\to \Hom(F_u,\omega_{R'}))\] 
  in $\Hom(F_u,\omega_R)$. Once we pick an isomorphism $F_u\cong R'$, we obtain isomorphism 
  \[\Hom(F_u,\omega_{R'})\cong \omega_{R'}\cong \x^{\mathbf{1}} R'.\]
   Write $\z_0$ for the image of $\x^{\mathbf{1}}$ under the surjection $\x^{\mathbf{1}}R'\to \omega_{A'}$.  
\end{remark}

Set
\[
\fld\coloneqq \bbF_p\left(\{\theta_{i,j}\mid 1\leq i\leq d,\ 1\leq j\leq n\}\cup \{\theta_j\mid 1\leq j\leq n\}\right),\quad
\ell_i\coloneqq \sum_j \theta_{i,j}x_j,\quad 
\ell\coloneqq \sum_j \theta_j x_j.
\]
Now $(\ell_1,\dots,\ell_d)$ is a regular linear system of parameters for $R\coloneqq R'\otimes_{\bbF_p} \fld$. Indeed, for $0\leq h\leq d$, set
\[\fld_h\coloneqq \bbF_p\left(\{\theta_{i,j}\mid 1\leq i\leq h,\ 1\leq j\leq n\}\cup \{\theta_j\mid 1\leq j \leq n\}\right).\]
Because each $\theta_{i,j}$ is transcendental over $k_{i-1}$, $\ell_i$ avoids the associated primes of $(R'\otimes \fld_{i-1})/(\ell_1,\dots,\ell_{i-1})$, and thus of $(R'\otimes \fld_i)/(\ell_1,\dots,\ell_{i-1})$. Hence, $\ell_i$ is not a zero divisor by the discussions in  \cite[1.5.10--1.5.12]{BrunsHerzog}.  

Set $R\coloneqq R'\otimes \fld$ and $A\coloneqq R/I$ where $I$ is the ideal of $R$ generated by $I'$. Let $\AA\coloneqq A/(\ell_1,\dots,\ell_d)$ so that $\omega_{\AA}\cong\omega_{A}/(\ell_1,\dots,\ell_d)$. Write $s$ for the socle degree of $\AA$.

\begin{prop} \label{p:anisotropyquotient} Suppose $R/I$ satisfies the above hypotheses with a $p$-conducting Parseval core $H_0$. Let $\w\in A^i$ be nonzero. 
Let $k\in\Z_{\geq 0}$ satisfy $pi+k\leq s$. Then, $\ell^k \w^p\neq 0$ in $\AA$. 
\end{prop}

\begin{proof}
Let $\z\in R\cong \Hom(F_u,\omega_R)$ descend to a nonzero element of $(\omega_{\AA})_{i-s}$ for $i \leq \frac{s}{p}$. Write $\z=\w\z_0$ for $\w\in R$. 
We will apply the Parseval--Rayleigh identity, 
\[\Vol(\z)=\Vol(\Phi^*_R(\ell_1^{p-1}\dots\ell_d^{p-1}H_0\z))^p.\]

  We first prove this for $k\leq p-1$.
  Since $\w\z_0\neq 0$, by duality between $\AA$ and $\omega_{\AA}$, there is $\t\in R^{s-i}$  such that $\t\w\z_0\neq 0$ in $(\omega_A)_0$.  We can suppose $\t$ is a monomial in $R'$. 
  Similarly, we can find
\begin{enumerate}
\item A monomial $\h$ in $H_0$,
\item $i_{q,r}\in [n]$ for $q=1,\dots,d$ and $r=1,\dots,p-1$
\item $j_q\in [n]$, for $q=1,\dots,k$, and
\item  A monomial $\x'\in R^{s-pi-k}$ such that
\end{enumerate}
\[(x_{i_{1,1}}\dots x_{i_{1,p-1}})\dots(x_{i_{d,1}}\dots x_{i_{d,p-1}})x_{j_1}\dots x_{j_k}\h\x'=c\x^{(p-1)\mathbf{1}}\t^p\]
for some $c\in \bbF_p^*$. Indeed by the $p$-conducting, there is a monomial $\h$ such that $p\deg(\t)+(p-1)\mathbf{1}-\deg(\h)\in\Z_{\geq 0}^n$. It is straightforward to make the other choices.
  
Let $S=S_1=\dots=S_d=\{x_1,\dots,x_n\}$. There exists $B\in \cB$ and $b\colon S\to \Z_{\geq 0}$ with $\sum_{\y'\in S} b(\y')=k$ such that
\[\x^B=(x_{i_{1,1}}\dots x_{i_{1,p-1}})\dots(x_{i_{d,1}}\dots x_{i_{d,p-1}}),\quad \prod_{\x\in S} \x^{b(\y)}=x_{j_1}\dots x_{j_k}.\]
  
  Write $\partial^b\coloneqq \prod_{\y\in S} \partial_{\theta_{\y}}^{b(\y)}$, $b!\coloneqq \prod_{\y\in S} b(\y)!$, so
\[  \partial^B\partial^b\Vol(\ell^k\x'\w^pz_0)=b!\partial^B\Vol(x_{j_1}\dots x_{j_k}\x'\w^p\z_0)\\=b!\Vol(\Phi_R^*(H_0\x^Bx_{j_1}\dots x_{j_k}\x'\w^p\z_0))^p\]
Now, only the monomial $\h$ in $H_0$ contributes to the volume since if another monomial $\h'$ did, then the difference of their fine degrees would be divisible by $p$. Hence, this quantity equals
 \[b!\Vol(\Phi_R^*(c\x^{(p-1)\mathbf{1}}\t^p\w^p\z_0))^p=c b!\Vol(\t\w \z_0)^p\neq 0\]
where we used $k\leq p-1$ to conclude $b(\y)!\neq 0$, and hence $b!\neq 0$.
  
For the general case, suppose by way of contradiction that $k\leq s-pi$  is the minimal value for which $\ell^k\w^p=0$. By the division algorithm, write $k=pb+k'$ with $0\leq k'\leq p-1$. Then,
\[0=\ell^k\w^p=\ell^{k'}(\ell^b\w)^p.\]
From this we conclude $\ell^b\w=0$ contradicting $\ell^b\w^p\neq 0$. 
\end{proof}

For $p=2$, we immediately conclude that  for $\w\in \AA^i$  nonzero, we have $\ell^{s-2i}\w^2\neq 0$. Therefore, $\ell^{s-2i}\w\neq 0$.

\subsection{Artinian reductions of almost Gorenstein rings}

In this section, we will prove that in some cases, the Parseval--Rayleigh identity implies an anisotropy statement, which when $p=2$ implies the generic Hard Lefschetz theorem in the Gorenstein setting and a weaker Lefschetz theorem in the Cohen--Macaulay setting. We will have to put the following conditions on $R'$:
\begin{enumerate}
\item $R'$ is a $d$-dimensional graded Cohen--Macaulay $\bbF_p$-algebra, generated in degree $1$,
\item the canonical module $\omega_{R'}$ is isomorphic to an ideal in $R'$, and
\item under the ring structure on $\omega_{R'}$ induced by the inclusion $\omega_{R'}\hookrightarrow R'$ as above, $\Phi_{R'}^*(\z^p)=\z$ for any $\z\in \omega_{R'}$.
\end{enumerate}

\begin{remark} The last two conditions are quite strong. They are related to a ring being generically Gorenstein \cite[Proposition~3.3.18]{BrunsHerzog} but they force $\omega_R$ to sit inside $R$ in a very specific way. When $R'$ is the Stanley--Reisner ring of a complex $\Sigma$, it constrains the combinatorics. 
\end{remark}

Let $S\subset ({R'})^1$ be a finite spanning set for $(R')^1$. Let $m_1,\dots,m_d$ be positive integers. For each $i$, write $S_i\coloneqq S^{\cdot m_i}\subset (R')^{m_i}$ be the set of all elements that can be written as a product of $m_i$ elements of $S$. Set
\[
\fld\coloneqq \bbF_p\left(\{\theta_{i,\y}\mid 1\leq i\leq d,\ \y\in S_i\}\cup \{\theta_\y\mid \y\in S\}\right),\quad
\ell_i\coloneqq \sum_{\y\in S_{m_i}}\theta_{i,\y}\y,\quad
\ell\coloneqq \sum_{\y\in S}\theta_{\y}\y,
\]
so $(\ell_1,\dots,\ell_d)$ is a regular homogeneous system of parameters for $R\coloneqq R'\otimes_{\bbF_p} \fld$ as above.

Set $A\coloneqq R/(\ell_1,\dots,\ell_d)$ so that $\omega_A\cong\omega_{R}/(\ell_1,\dots,\ell_d)$. Write $s$ for the top degree of $A$, so $\omega_A$ is graded in degrees $-s,\dots,0$. The following appeared in \cite{AOPP:complete}. Its proof is similar to that of Proposition~\ref{p:anisotropyquotient} when $H_0=1$.

\begin{prop} \label{p:anisotropy} Suppose $R$ satisfies the above hypotheses.
Let $\z\in (\omega_A)_{i-s}$ be nonzero for $i \leq \frac{s}{p}$. Let $k\in\Z_{\geq 0}$ satisfy $pi+k\leq s$. Then, $\ell^k \z^p\neq 0$.
\end{prop}

For $p=2$, we immediately conclude that  for $\z\in (\omega_A)_{i-s}$  nonzero, we have $\ell^{s-2i}\z^2\neq 0$. Therefore, $\ell^{s-2i}z\neq 0$.

\section{Anisotropy for Generic Reductions of rings associated to Pfaffian ideals} \label{s:pfaffians}

In this section, we will study rings associated to  Pfaffian ideals,
which
is the next simplest class of Gorenstein ideals after the complete
intersections studied in~\cite{AOPP:complete}. By the Buchsbaum-Eisenbud
theorem \cite[Theorem~3.4.1]{BrunsHerzog} every Gorenstein  codimension
$3$ homogeneous ideal $I$  is given as the maximal Pfaffians
of a skew-symmetric $m \times m$ matrix $A$, where $m \geq 3$ is an odd
integer. Here, we will study the Pfaffians of a skew-symmetric matrix $A$ whose entries are indeterminates in characteristic $2$.

Let $m \geq 3$ be an odd integer, and consider  the polynomial ring 
    \[
           R \coloneqq  \bbF_2 [ z_{ij} :  1 \leq i < j \leq m]
    \]
    with the standard grading.    We write $z_{ji} = z_{ij}$ for  $1 \leq i < j \leq m$ and $z_{ii}=0$.      
    
     We consider the $m \times m$ alternating matrix $A$ whose  $(i,j)$-entry is $z_{ij}$. We denote by 
     $P_i$ the Pfaffian of the  $(m-1) \times (m-1)$   submatrix of $A$ obtained
     by deleting the  $i$-th column and the $i$-th row of $A$. Let $I$ be the ideal generated by $P_1,\dots,P_m$. There is a minimal free resolution $F_\bullet$ of $R/I$ over $R$ given by
     \[\xymatrix{
     0\ar[r]&R(-m)\ar[r]^{P}&F(-\frac{m+1}{2})\ar[r]^{A}&F(-\frac{m-1}{2})\ar[r]^>>>>>{P^t}&R\ar[r]&0}\]
     where $F=R^m$ and  $P\colon R\to F$ be given by $y\mapsto (P_1 y,\dots,P_m y)$. We will use the Parseval--Rayleigh identity in characteristic $2$ for the generic reduction of $R/I$ to show that it obeys anisotropy, implying the Lefschetz property. We do this by finding an explicit description of the chain map from the Frobenius twist $\Phi^{\flat}_\bullet\colon\varphi^*F_\bullet\to F_\bullet$. The components of $\Phi^{\flat}_\bullet$ will be expressed as generating functions of multigraphs, showing the strong combinatorial character of our theory.

\subsection{Generating functions attached to multigraphs}
We will write elements in $R$ as generating functions of sets of multigraphs, interpreted as functions $w\colon E(K_m)\to \Z_{\geq 0}$.
Let $\cS(K_m)$ be the set of functions $w\colon E(K_m)\to \Z_{\geq 0}$. We will consider the {\em support} $\Supp(w)$ to be the multigraph formed by all the vertices of $K_m$ where there are $w(e)$ edges drawn between the endpoints of $e$. Attached to $w$ is the monomial 
\[z^w\coloneqq \prod_{e\in E(K_m)} z_e^{w(e)}.\]
where $z_e\coloneqq z_{ij}$ when $i$ and $j$ are the endpoints of $e$.
A {\em set of decorated multigraphs} is a pair $(T,\sigma)$ where $T$ is some finite set and $\sigma\colon T\to \cS(K_m)$ is a function, and we will view an element of $T$ as the multigraph $\sigma(T)$ with additional data. Attached to $(T,\sigma)$ is the generating function
\[F_{(T,\sigma)}\coloneqq \sum_{Y\in T} z^{\sigma(Y)}.\]
Observe that
\[F_{(T,\sigma)}=\sum_{w\in \cS(K_m)} |\sigma^{-1}(w)| z^w\\
=\sum_{w\in \cO_{(T,\sigma)}} z^w\]
where 
\[\cO_{(T,\sigma)}\coloneqq \{w\in \cS(K_m): |\sigma^{-1}(w)| \text{ is odd}\}.\]
The degree of $w$ at a vertex $v$ is $d_w(v)\coloneqq \sum_{e\ni v} w(e)$. We say $w$ is $d$-regular if $d_w(v)=d$ for all $v\in V(K_m)$.

The set of matchings on $[m]\setminus \{i\}$ can be viewed as a decorated multigraph in the following way: 
\[\cM([m]\setminus\{i\})\coloneqq \{w\in \cS(K_m)\mid d_w(i)=0,\ d_w(j)=1\text{ for all }j\neq i\};\]  
$\sigma\colon \cM([m]\setminus\{i\})\to \cS(K_m)$ is the usual inclusion. Then, $F_{\cM([m]\setminus\{i\})}=P_i$.

\begin{lemma} \label{l:FO}
For
\[H_0\coloneqq \sum_{i<j} z_{ij}P_iP_j,\]
we have the equality
\[H_0=F_{\mathcal{OC}}.\]
where $\mathcal{OC}\subseteq \cS(K_m)$ is the set of $2$-regular $w$ such that the components of $\Supp(w)$ are a single odd cycle and a collection of doubled edges.
\end{lemma}

\begin{proof}
  Let ${\cC}$ be the set of $2$-regular weighted multigraphs (so each edge has weight $1$ or $2$) where the weight $1$ edges are colored red, gold, and green such that
  \begin{enumerate}
    \item there is exactly one green edge,
    \item no vertex belongs to two edges of the same color, and
    \item if the green edge is $ij$ with $i<j$, then $i$'s other edge is gold and $j$'s other edge is red.
  \end{enumerate}
   Let $\sigma\colon \cC\to \cS(K_m)$ forget the coloring. Then, $F_{(\cC,\sigma)}=H_0$ where the green edge corresponds to $\ij$, and the red and gold edges correspond to $P_i$ and $P_j$, respectively. In any graph in $\cC$, the green edge belongs to an odd cycle of length $\ell$ through which there is an alternating red and gold path  between its endpoints such that the other edge containing its smaller endpoint is colored gold. There are also $c\geq 0$ alternating even cycles and some edges of multiplicity $2$. Given an element of $\sigma^{-1}(w)$, we can produce all other elements of $\sigma^{-1}(w)$ by moving the green edge to another edge of the odd cycles or interchanging the colors of some of the even cycles. Thus, 
  \[|\sigma^{-1}(w)|=\ell 2^c\equiv 
  \begin{cases}
  1 &\text{if }c=0\\
  0 &\text{otherwise,}
  \end{cases}\]
  and only contributions come from elements of $\mathcal{OC}$ which are each counted exactly once. Hence, $F_{\cC}=F_{\mathcal{OC}}$.
\end{proof}


\subsection{Frobenius twists of the Pfaffian resolution}

We have a chain map from the Frobenius twist of the Pfaffian resolution to the Pfaffian resolution:
\[\xymatrix{
0\ar[r]&R(-2m)\ar[r]^{P^{[2]}}\ar[d]^{\Phi_3^{\flat}}&F(-m-1)\ar[r]^{A^{[2]}}\ar[d]^{\Phi_2^{\flat}}&F(-m+1)\ar[r]^>>>>{(P^t)^{[2]}}\ar[d]^{\Phi_1^{\flat}}&R\ar[r]\ar[d]^{\Phi_0^{\flat}}&0\\
0\ar[r]&R(-m)\ar[r]^{P}&F(-\frac{m+1}{2})\ar[r]^{A}&F(-\frac{m-1}{2})\ar[r]^>>>>{P^t}&R\ar[r]&0.
}\]
Under the condition that $\Phi_0^{\flat}$ is the identity, $\Phi_\bullet^{\flat}$ is unique up to chain homotopy. To describe a choice of $\Phi^\flat_\bullet$,
we first define the differential operator
\[D_k\coloneqq \sum_{u<v\mid u,v\neq k} z_{uk}z_{kv}\frac{\partial}{\partial z_{uv}}\]
where $D_k(P_j)=F_{\cG_{jk}}$ can be interpreted as starting with a matching on $\cM([m]\setminus j)$, removing an edge $uv$ (not containing $j$) and connecting its endpoints to a vertex $k$. Therefore, when 
\begin{enumerate}
  \item for $j=k$, $\cG_{kk}$ consists of all graphs with $(m+1)/2$ edges of weight $1$ whose connected components are a single path of length $2$ with $k$ in the middle and $(m-3)/2$ edges;
  \item for $j\neq k$, $\cG_{jk}$ consists of all graphs with $(m+1)/2$ edges whose connected components are the singleton $j$, $(m-5)/2$ edges, and a trivalent vertex with $k$ at the center.
\end{enumerate}

\begin{figure}
\begin{tikzpicture}
  \node[circle, fill, inner sep=2pt, label=above:$k$] (k) at (1,1) {};
  \node[circle, fill, inner sep=2pt] (a) at (0,1) {};
  \node[circle, fill, inner sep=2pt] (b) at (2,1) {};

  \node[circle, fill, inner sep=2pt] (c) at (2.7,0) {};
  \node[circle, fill, inner sep=2pt] (d) at (2.7,1) {};
  \node at (3.3,.5) {$\dots$};

  \node[circle, fill, inner sep=2pt] (e) at (3.8,0) {};
  \node[circle, fill, inner sep=2pt] (f) at (3.8,1) {};
  
  \draw (a) -- (k) -- (b);
  \draw (c) -- (d);
  \draw (e) -- (f);
  \node[circle, fill, inner sep=2pt, label=above:$k$] (k) at (7,1) {};
  \node[circle, fill, inner sep=2pt] (a) at (6,1) {};
  \node[circle, fill, inner sep=2pt] (b) at (8,1) {};
  \node[circle, fill, inner sep=2pt] (g) at (7,0) {};

  \node[circle, fill, inner sep=2pt] (c) at (8.7,0) {};
  \node[circle, fill, inner sep=2pt] (d) at (8.7,1) {};
  \node at (9.3,.5) {$\dots$};
  \node[circle, fill, inner sep=2pt] (e) at (9.8,0) {};
  \node[circle, fill, inner sep=2pt] (f) at (9.8,1) {};

  \node[circle, fill, inner sep=2pt, label=above:$j$] (j) at (10.5,1) {};
  
  \draw (a) -- (k) -- (b);
  \draw (k) -- (g);

  \draw (c) -- (d);
  \draw (e) -- (f);
\end{tikzpicture}
\caption{Elements of $\cG_{kk}$ and $\cG_{jk}$}
\end{figure}

\begin{theorem} 
The following choice of $\Phi^\flat_\bullet\colon \varphi_*F_\bullet \to F_\bullet$ gives a chain map:
\begin{alignat*}{4}
\Phi_0^{\flat}&=\operatorname{Id_R},&&
(\Phi_1^{\flat})_{ij}&&=\begin{cases}
P_i&\text{if }i=j\\
0&\text{otherwise,}
\end{cases}\\
(\Phi_2^{\flat})_{ij}&=D_j(P_i),\quad
&&\Phi_3^{\flat}&&=H_0.
\end{alignat*}
\end{theorem}

\begin{proof}
  The commutativity of the right square is obvious.
  
  For the middle square, we must show $\Phi_1^{\flat} A^{[2]}=A\Phi_2^{\flat}$, which means that for all $(i,k)$, 
  \[P_iz_{ik}^2=\sum_{j\neq i} z_{ij}D_k(P_j).\]
  The exactness of $F_\bullet$ implies $P^tA=0$ which is equivalent to $\sum_j z_{ij}P_j=0$. Applying $D_k$ yields  
  $\sum_{j\neq i} z_{ij}D_k(P_j)+\sum_{j\neq i} D_k(z_{ij})P_j=0,$
  which implies
  \begin{align*}
   \sum_{j\neq i} z_{ij}D_k(P_j)&=\sum_{j\neq i} D_k(z_{ij})P_j=\sum_{j\neq i} z_{ik}z_{jk}P_j=\sum_{j\neq i} z_{ik} z_{jk}\sum_{M\in \cM([m]\setminus j)} z^M\\
   &=z_{ik}\sum_{j\neq i}  z_{jk}\left(\sum_{\substack{M\in \cM([m]\setminus j)\\ik\in M}} z^M\right)+z_{ik}\sum_{j\neq i}  z_{jk}\left(\sum_{\substack{M\in \cM([m]\setminus j)\\ik\not\in M}} z^M\right)\\
   &=z_{ik}^2\sum_{j\neq i}  z_{jk}\left(\sum_{M\in \cM([m]\setminus \{i,j,k\})} z^M\right)+z_{ik}\sum_{j\neq i}  z_{jk}\left(\sum_{h\neq i,j,k}z_{hk} \sum_{M\in \cM([m]\setminus \{h,j,k)\}} z^M\right)\\
   &=z_{ik}^2\sum_{j\neq i} \left(\sum_{\substack{M\in \cM([m]\setminus i)\\jk\in M}} z^M\right)+z_{ik}\sum_{\substack{j,h\not\in \{i,k\}\\j\neq h}} z_{jk}z_{hk}\left(\sum_{M\in \cM([m]\setminus \{h,j,k)\}} z^M\right)\\
   &=z_{ik}^2P_i
  \end{align*}
  where the last equality follows from the evenness of all terms in the second summand due to invariance of the sum under interchange of $h$ and $j$.

 For the left square, we must show that for any $i$,
  \[\sum_j D_j(P_i)P_j^2=P_i H_0.\]
We write the left side as $D_i(P_i)P_i^2+\sum_{j\neq i} D_j(P_i)P_j^2$. Observe that $P_j^2$ counts matchings on $[m]\setminus j$ where each edge is given multiplicity $2$.

Now, $D_i(P_i)P_i^2$ corresponds to a set $\cE_{ii}$ whose elements are multigraphs on $[m]$ with edges weightings of $0,1,2,3$ with $\deg(v)=3$ for $v\neq i$ and $\deg(i)=2$ such that every vertex except $i$ belongs to at least one edge of weight greater than $2$ and $i$ belongs to two edges of weight $1$. This observation is seen by considering all possible multigraphs that arise as the sum of an element of $\cG_{ii}$ and of an element of $\cM([m]\setminus i)$ with multiplicities doubled.
Thus, a multigraph in $\cE_{ii}$ has the following components: an odd cycle containing $i$; possibly some even cycles; and possibly some edges of multiplicity $3$. All the even cycles must be alternating in the sense that the weights on its edges alternate between $1$ and $2$. Similarly,  the edges on the odd cycle alternate between weights $1$ and $2$ except at $i$ where they are both $1$. This is depicted in Figure~\ref{fig:eii}

On the other hand, for $i\neq j$, $D_j(P_i)P_j^2$  corresponds to a set $\cE_{ij}$ which consists of multigraphs on $[m]$ with edges weightings of $0,1,2,3$ with $\deg(v)=3$ for $v\neq i$ and $\deg(i)=2$ such that every vertex except $i$ and $j$ belongs to exactly one edge of weight $1$. This is seen by considerings sums of an elements of $\cG_{ij}$ and an element of $\cM([m]\setminus j)$ with multiplicities doubled. Here, $i$ belongs to one edge of weight $2$ and $j$ belongs to three edges of weight $1$. Thus, a multigraph in $\cE_{ij}$ has the following components: possibly some even alternating cycles;  possibly some edges of multiplicity $3$; and exactly only component that is 
an odd cycle containing $j$ union an alternating path from $j$ to $i$. This is depicted in Figure~\ref{fig:eij}

Now, $P_iH_0=P_iF_{\mathcal{OC}}$ by Lemma~\ref{l:FO}. Let $\cG_i$ be the set consisting of pairs $(\Gamma,M)$ where $\Gamma$ is a $2$-regular multigraph whose components are an odd cycle and some edges of multiplicity $2$ while $M$ is a matching on $[m]\setminus i$. Now, $\Gamma\cup M$, considered as a multigraph has degree $2$ at $i$ and degree $3$ at all other vertices. A vertex $v\neq i$ can have the following possible adjacencies: one edge of weight $3$; an edge of weight $2$ and an edge of weight $1$; or three edges of weight $1$. We call vertices belonging to three weight $1$ edges {\em trivalent}, and they must all belong to the odd cycle of $\Gamma$. Now, if $v$ belongs to an edge of weight $2$, we can travel along that edge through an alternating path until one of the following happens:
\begin{enumerate}
  \item the path closes up at $v$, giving an alternating even cycle,
  \item the path meets $i$; if $i$ belongs to two edges of multiplicity $1$, then $i$ belongs to the odd cycle; if $i$ belongs to one edge of multiplicity $2$, the path terminates at $i$, or
  \item the path meets a trivalent vertex which also must belong to the odd cycle.
\end{enumerate}
Thus, $\Gamma\cup M$ has the following components: edges of weight $3$; alternating even cycles;  and a {\em nontrivial component} $K$ containing a distinguished odd cycle $C$. If $i$ belongs to an edge of multiplicity $2$, there is a path $T$ from a vertex $j\in V(C)$ to $i$. Otherwise, $i$ is a vertex of the odd cycle and we set $T\coloneqq \varnothing$ and set $j\coloneqq i$.  Now, $K$ is obtained from $C\cup T$ by adding alternating paths (beginning and ending with edges of multiplicity $1$) between pairs of vertices of $C$ different from $j$. We call these {\em alternating chords}. 

Let $\cF_i$ be the set of pairs $(\Delta,C)$ where $\Delta$ is a multigraph that has degree  $2$ at $i$ and degree $3$ at all other vertices and $C$ is a distinguished odd cycle in the support of $\Delta$ that contains $i$ and all the trivalent vertices of $\Delta$.
Now, the assignment $(\Gamma,M)\mapsto (\Gamma\cup M,C)$  gives a bijection $\cG_i\to \cF_i$.
We claim that multigraphs with alternating chords do not contribute to $F_{\cG_i}$. To prove this, we have to show that if $(\Delta,C)\in \cF_i$ has an alternating chord, then there are an even number of odd cycles in $\Delta$ containing $j$ and all the trivalent vertices. Let $K$ be the component of $\Delta$ containing $C$. Form the cubic graph $G_{\Delta}$ by replacing alternating paths in $K$ by edges. If $i\neq j$ contract the edge from $i$ to $j$. Now, $j$ has degree $2$. Replace the two edges containing $j$ by a single edge $e$ to obtain a graph $G_{\Delta}$. The cycles in $\Delta$ passing through $i$ and all the trivalent vertices correspond exactly to Hamiltonian cycles in $G_\Delta$ containing the edge $e$. Because there are an even number of such cycles by C.~A.~B.~Smith's theorem \cite{THamiltonian}, graphs with alternating chords do not contribute to $F_{\cF_i}$. Thus, the only contributors from $\cF_i$ (which all contain a unique odd cycle) are in bijective correspondence with elements of $\sqcup_{j} \cE_{ij}$. Hence,
\[P_iH_0=F_{\cF_i}=F_{\cG_i}=\sum_j F_{\cE_{ij}}=\sum_j D_j(P_i)P_j^2.\qedhere\]
\end{proof}

We have the following Parseval--Rayleigh identity as an immediate consequence:
\begin{corollary}
Let $\AA$ be the generic Artinian reduction of $A\otimes\fld$ where the quotient is taken by generic linear forms $\ell_1,\dots,\ell_t$ as in Section~\ref{ss:quotientsofpolynomialrings} for $t=\binom{m}{2}-3$.
Then, for $\z\in (\omega_{\AA})_0$,
$\Vol(\z)=\Vol(\Phi^*_R(\ell_1\dots\ell_t H_0\z))^p$
\end{corollary}
This identity can be written more explicitly following  Example~\ref{e:parsevalinpolynomialrings} and Proposition~\ref{p:parsevalartinian}.

\begin{figure}
\begin{tikzpicture}[
  scale=1,
  v/.style={circle, fill, inner sep=2pt},
  every node/.style={font=\small},
  edge label/.style={midway, fill=white, inner sep=1pt}
]

\node[v,label=above:$i$] (i) at (90:1.3) {};
\node[v] (a) at (18:1.3) {};
\node[v] (b) at (-54:1.3) {};
\node[v] (c) at (-126:1.3) {};
\node[v] (d) at (162:1.3) {};

\draw (i) -- node[edge label, above right] {$1$} (a);
\draw (a) -- node[edge label, right] {$2$} (b);
\draw (b) -- node[edge label, below] {$1$} (c);
\draw (c) -- node[edge label, left] {$2$} (d);
\draw (d) -- node[edge label, above left] {$1$} (i);

\begin{scope}[xshift=4cm]
  \foreach \ang/\name in {90/e1,30/e2,-30/e3,-90/e4,-150/e5,150/e6}
    \node[v] (\name) at (\ang:1.1) {};

  \draw (e1) -- node[edge label, above right] {$1$} (e2);
  \draw (e2) -- node[edge label, right] {$2$} (e3);
  \draw (e3) -- node[edge label, below right] {$1$} (e4);
  \draw (e4) -- node[edge label, below left] {$2$} (e5);
  \draw (e5) -- node[edge label, left] {$1$} (e6);
  \draw (e6) -- node[edge label, above left] {$2$} (e1);
\end{scope}

\begin{scope}[xshift=6.5cm]
  \node[v] (f1) at (0,0) {};
  \node[v] (f2) at (1,0) {};
  \node[v] (f3) at (1,1) {};
  \node[v] (f4) at (0,1) {};

  \draw (f1) -- node[edge label, below] {$1$} (f2);
  \draw (f2) -- node[edge label, right] {$2$} (f3);
  \draw (f3) -- node[edge label, above] {$1$} (f4);
  \draw (f4) -- node[edge label, left] {$2$} (f1);
\end{scope}

\begin{scope}[xshift=8.8cm,yshift=0cm]
  \node[v] (u) at (0,1) {};
  \node[v] (w) at (1,1) {};
  \draw (u) -- node[edge label, above] {$3$} (w);
\end{scope}

\begin{scope}[xshift=10.5cm,yshift=0cm]
  \node[v] (p) at (0,1) {};
  \node[v] (q) at (1,1) {};
  \draw (p) -- node[edge label, above] {$3$} (q);
\end{scope}

\end{tikzpicture}
\caption{Possible graph in $\cE_{ii}$}
\label{fig:eii}
\end{figure}

\begin{figure}
\begin{tikzpicture}[
  scale=1,
  v/.style={circle, fill, inner sep=2pt},
  every node/.style={font=\small},
  edge label/.style={midway, fill=white, inner sep=1pt}
]

\node[v,label=above:$j$] (j) at (90:1.3) {};
\node[v] (a) at (18:1.3) {};
\node[v] (b) at (-54:1.3) {};
\node[v] (c) at (-126:1.3) {};
\node[v] (d) at (162:1.3) {};

\draw (j)--(a)--(b)--(c)--(d)--(j);

\draw (j) -- node[edge label, above right] {$1$} (a);
\draw (a) -- node[edge label, right] {$2$} (b);
\draw (b) -- node[edge label, below] {$1$} (c);
\draw (c) -- node[edge label, left] {$2$} (d);
\draw (d) -- node[edge label, above left] {$1$} (i);

\node[v] (p1) at (1.7,1.6) {};
\node[v] (p2) at (2.6,1.6) {};
\node[v] (p3) at (3.6,1.6) {};
\node[v,label=right:$i$] (i) at (4.5,1.6) {};

\draw (j)  -- node[edge label, above right] {$1$} (p1);
\draw (p1) -- node[edge label, below] {$2$} (p2);
\draw (p2) -- node[edge label, above] {$1$} (p3);
\draw (p3) -- node[edge label, below right] {$2$} (i);

\begin{scope}[xshift=5.8cm,yshift=0cm]
  \foreach \ang/\name in {90/e1,30/e2,-30/e3,-90/e4,-150/e5,150/e6}
    \node[v] (\name) at (\ang:1.1) {};

  \draw (e1) -- node[edge label, above right] {$1$} (e2);
  \draw (e2) -- node[edge label, right] {$2$} (e3);
  \draw (e3) -- node[edge label, below right] {$1$} (e4);
  \draw (e4) -- node[edge label, below left] {$2$} (e5);
  \draw (e5) -- node[edge label, left] {$1$} (e6);
  \draw (e6) -- node[edge label, above left] {$2$} (e1);
\end{scope}

\begin{scope}[xshift=7.8cm,yshift=0cm]
  \node[v] (f1) at (0,0) {};
  \node[v] (f2) at (1,0) {};
  \node[v] (f3) at (1,1) {};
  \node[v] (f4) at (0,1) {};

  \draw (f1) -- node[edge label, below] {$1$} (f2);
  \draw (f2) -- node[edge label, right] {$2$} (f3);
  \draw (f3) -- node[edge label, above] {$1$} (f4);
  \draw (f4) -- node[edge label, left] {$2$} (f1);
\end{scope}

\begin{scope}[xshift=9.3cm,yshift=0cm]
  \node[v] (u) at (0,1) {};
  \node[v] (w) at (1,1) {};
  \draw (u) -- node[edge label, above] {$3$} (w);
\end{scope}

\begin{scope}[xshift=10.8cm,yshift=0cm]
  \node[v] (r) at (0,1) {};
  \node[v] (s) at (1,1) {};
  \draw (r) -- node[edge label, above] {$3$} (s);
\end{scope}

\end{tikzpicture}
\caption{Possible graph in $\cE_{ij}$}
\label{fig:eij}
\end{figure}

\subsection{Anisotropy properties}

\begin{theorem}
The Parseval core $H_0$ is $2$-conducting. Thus, there is a field extension $\fld$ of $\bbF_2$ such that the generic Artinian reduction $\AA$ of the quotient $R/I$ satisfies anisotropy.
\end{theorem}

\begin{proof}
Let $n=\binom{m}{2}$, the number of edges of $K_m$. The fine degree of a monomial in $H_0$ corresponding to a function $w\colon E(K_m)\to \Z_{\geq 0}$  is simply $w$, considered as an element of $\Z_{\geq 0}^n$.   To prove property (\ref{i:smallermonomial}) of $2$-conducting, observe that for any $\mathbf{a}$, we may choose any element  $\mathbf{m}$ of $\Z_{\geq 0}^n$ corresponding to a Hamilton cycle $C$. 

We claim that such an $\mathbf{m}$ satisfies property (\ref{i:paritynonoverlapping}). If $\mathbf{m'} \in \operatorname{supp}(H_0)$, then $\mathbf{m'}$ corresponds to a multigraph with $q$ doubled edges and an odd cycle $C'$ of length $m-2q$. We claim that if $C\neq C'$, the symmetric difference $C\triangle C'$ has at least $4$ edges. Observe that it cannot have $2$ since if $(C\setminus e)\cup\{e'\}$ is a  cycle, we must have $e=e'$. Moreover, $|C\triangle C'|\equiv |C|+|C'|\ (\operatorname{mod} 2)$, so $|C\triangle C'|$ must be even.
Now, $\mathbf{m}-\mathbf{m'}$ has odd multiplicity on the edges in $C\triangle C'$, and hence cannot be divisible by $2$.
\end{proof}

\section{Lattice polytopes and simplicial spheres} \label{s:latticepolytopes}

We explain how to deduce the some of the results of \cite{PP} and \cite{APPlattice} from our viewpoint, itself inspired by \cite{APPlattice}. We restrict to characteristic $2$.

\begin{theorem}[The $g$-theorem for simplicial spheres \cite{AHL,PP}] Let $\Sigma$ be a $(d-1)$-dimensional simplicial sphere.  Then, in the Stanley--Reisner ring $R\coloneqq \fld[\Sigma]$ over some field $\fld$, there is a linear system of parameters $\ell_1,\dots,\ell_d\in R^1$ such that for $A\coloneqq R/(\ell_1,\dots,\ell_d)$, the canonical module $\omega_{A}\coloneqq \omega_R/(\ell_1,\dots,\ell_d)$ has the Hard Lefschetz property with respect to some $\ell\in A^1$.
\end{theorem}

\begin{proof}
Set $R'\coloneqq \bbF_2[\Sigma]$.
  Let $S=S_1=\dots=S_d=\{x_v\mid v\in\Sigma\}$, the degree $1$ generators of $\fld[\Sigma]$, and let $\fld$ be as above. By the computation of $\omega_{R'}\cong H_{\m}(R)^\vee$ in say \cite[Section~5.3]{BrunsHerzog}, using the \v{C}ech complex, $\omega_R$ can be identified with $R$. Moreover, the $p$-power map $\varphi$ induces a Frobenius trace $\Phi_{R'}$ on $\omega_{R'}$ by duality. This map is given by 
  \[\Phi_{R'}(\x^\a)=
  \begin{cases}
    \x^{\a/p} &\text{if }p\mid \a\\
    0 &\text{otherwise}
  \end{cases} \]
 for $\x^\a\in R'$. Clearly $\Phi_{R'}(\z^p)=\z$.
  Take the generic choice of $\ell_i$ and $\ell$ as above.
   
  By Proposition~\ref{p:anisotropy}, the multiplication map $\ell^{d-2i}\cdot \colon(\omega_A)_{i-d}\to (\omega_A)_{-i}$ is injective. Since $A$ is Gorenstein, the domain and target have the same dimensions and the map is an isomorphism. 
\end{proof}

Recall that for a lattice polytope $P\subset\R^d$, the cone $CP$ is defined by $CP\coloneqq \operatorname{Cone}(P\times\{1\})\subset\R^{d+1}$. The polytope $P$ is {\em IDP} if $CP$ is generated as a semigroup by the lattice points in $P\times\{1\}$. We write $\fld[CP]$ for the semigroup algebra on $CP\cap\Z^{d+1}$.

\begin{theorem}[\cite{APPlattice}] Let $P$ be a $d$-dimensional IDP polytope. Then, in the semigroup algebra $R\coloneqq \fld[CP]$ over some characteristic $2$ field $\fld$, there is a linear system of parameters $\ell_1,\dots,\ell_{d+1}\in R^1$,  such that for  $A\coloneqq R/(\ell_1,\dots,\ell_{d+1})$, the map $\ell\cdot\colon (\omega_{A})_{i-d}\to (\omega_{A})_{i-d+1}$ is injective for $i<\frac{d}{2}$ for some $\ell\in A^1$. Consequently, the $h^*$ vector of $P$ satisfies
\[h^*_{d}\leq h^*_{d-1}\leq \dots \leq h^*_{\lceil d/2\rceil}.\]
\end{theorem}

\begin{proof}
  Let $R'=\bbF_2[CP]$. By the computation of the canonical module in say \cite[Section~6.3]{BrunsHerzog}, by means of the complex $L^\bullet$, we obtain 
  $\omega_{R'}\cong \bbF_2[CP^\circ]$, the ideal in $\bbF_2[CP]$ generated by lattice points in the interior of $CP$. As in the Stanley--Reisner case, $\Phi_{R'}(\z^p)=\z$. The conclusion
   follows from Proposition~\ref{p:anisotropy} with $S=S_1=\dots=S_{d+1}$ equal to a set of degree $1$ generators of $\bbF_2[CP]$.
\end{proof}

The symmetry of the $h^*$-vector for reflexive polytopes immediately yields the following:

\begin{corollary}[The Ohsugi--Hibi Conjecture \cite{APPlattice}] If $P$ is a IDP reflexive polytope, then the $h^*$-vector is unimodal.
\end{corollary}



\begin{remark} These arguments applied to a generic homogeneous system of parameters $\ell_1,\dots,\ell_d$ in the Gorenstein ring $\fld[x_1,\dots,x_d]$ also give the Hard Lefschetz property for $\fld[x_1,\dots,x_d]/(\ell_1,\dots,\ell_d)$ \cite{AOPP:complete}.
This Hard Lefschetz for generic complete intersections also follows from that of any example as noted in \cite[Example~1.2]{DimcaIlardi}. The example of  $\fld[x_1,\dots,x_n]/(x_1^{a_1},\dots,x_n^{a_n})$ is there attributed to Stanley \cite{Stanley:Hilbertfunctions} and Watanabe \cite{Watanabe}.
\end{remark}
%
%
%

\section{Additional examples} \label{s:additional}

In this section, we consider some additional Gorenstein examples with minimal free resolutions $F_\bullet$. We find explicit expressions for the top piece of the chain map $\Phi^{\flat}_\bullet\colon\varphi^*F_\bullet \to F_\bullet$, which is described by polynomials, each denoted by $H_0$. While the chain map is only unique up to chain homotopy, we expect these $H_0$'s to have interesting combinatorial content as in Section~\ref{s:pfaffians}.

\subsection {The ring attached to Pfaffians of a $(5\times 5)$-matrix in characteristic $3$}   \label{sec!parsevalsinchar3}

We  use the  product of disjoint cycles notation for the elements of $\mathfrak{S}_5$, the symmetric group on $\{1,\dots,5\}$.
Let 
\[P_3\coloneqq \{(a,b,c) \mid  (a,b,c)  \text { is a permutation of }  (3,4,5)\}\]
 and consider the bijection \cite{GP},
\[
    \Psi\colon  P_3  \to   \{\text{order $20$  subgroups of }  \mathfrak{S}_5\},\quad  \Psi  ( a,b,c) =   \langle (1,2,a,b,c), (2,a,c,b) \rangle.
\]

Let  $\mathbbm{k}$ have characteristic $p=3$, and consider the polynomial ring  $R\coloneqq\mathbbm{k}[z_{i,j}]$ 
in $10$ variables $z_{i,j}$ with  $1 \leq i  <  j \leq  5$.
For $j < i$ we set $z_{i,j} = - z_{j,i}$. Let 
$\mathfrak{S}_5$ act
on $\mathbbm{k}[z_{i,j}]$  by $\sigma \cdot z_{i,j}  = z_{\sigma(i), \sigma(j)}$.
For a polynomial $\eta$ we denote by  $Q(\eta)$ the sum of the elements of the $\mathfrak{S}_5$-orbit of $\eta$. 

We denote by $A$ the $5 \times 5$ skew-symmetric matrix with $(i,j)$-entry equal to
$z_{i,j}$ when $i < j$ and equal to $0$ when $i=j$.  Let $I\subset R$ be the Gorenstein ideal generated by the $(4\times 4)$-Pfaffians of $A$.

Consider the following five degree $10$ monomials
\begin{align*}
  \eta_1  & =   z_{1,3}z_{1,4}z_{1,5}^2z_{2,3}z_{2,4}^2z_{2,5}z_{3,4}z_{3,5},   &  
             \eta_2  &  =    z_{1,3}z_{1,4}z_{1,5}^2z_{2,3}z_{2,4}z_{2,5}^2z_{3,4}^2,  \\
       \eta_3   & =    z_{1,4}^2z_{1,5}^2z_{2,3}^2z_{2,5}^2z_{3,4}^2,  & 
       \eta_4  &  =    z_{1,4}z_{1,5}^3z_{2,3}^2z_{2,4}z_{2,5}z_{3,4}^2,  \\
      \eta_5 & =   z_{1,5}^4z_{2,3}^2z_{2,4}^2z_{3,4}^2.       &   &        
\end{align*}


Assume $1 \leq t \leq 5$. 
Then  $\eta_t$ is balanced, in the sense that 
if we substitute  $a_ia_j$ for $z_{i,j}$ for all $i,j$ then the resulting monomial is
either  $\; \prod_{s=1}^5 a_s^4$ or  $\; -\prod_{s=1}^5 a_s^4$.

For  $(a,b,c) \in P_3$, consider the balanced monomial
\[
     \theta_{(a,b,c)} =  z_{2,a}^3 z_{b,c}^3 \prod_{i=2}^5 z_{1,i}.
\]  
We denote by  $U_{(a,b,c)}$ the  sum of the elements  of  the orbit of $\theta_{(a,b,c)}$
under the action of  the subgroup  $\Psi (a,b,c)$  of   $\mathfrak{S}_5$.

\begin{prop}  \label{[prop!klaa}
The cardinality $6$ set 
$\{U_{(a,b,c)} \mid (a,b,c) \in P_3\}$
 is stable under the action of       $\mathfrak{S}_5$, consists of a single 
$\mathfrak{S}_5$-orbit, and the stabilizer subgroup of the element  $U_{(a,b,c)}$ is the
order $20$ subgroup $\Psi (a,b,c)$  of   $\mathfrak{S}_5$.
 \end{prop}  

\begin{proof}   The proof was given by Macaulay2. The code is contained in the file \cite{AKOPP52}.
\end{proof}

\begin{example}  We have 
\[
   U_{(3,4,5)} = L_1 + L_2 + L_3 + L_4,  
\]
where   
 \[
    L_1 =  z_{1,2}z_{1,5}^{3}z_{2,3}z_{2,4}z_{2,5}z_{3,4}^{3}+z_{1,3}z_{1,4}^{3}z_{2,3}z_{2,5}^{3}z_{3,4}z_{3,5}+z_{1,2}z_{1,3}z_{1,4}z_{1,5}z_{2,4}^{3}z_{3,5 }^{3}
 \]     
  \[
    L_2 =  z_{1,2}z_{1,4}^{3}z_{2,3}z_{2,4}z_{2,5}z_{3,5}^{3}+z_{1,4}z_{1,5}^{3}z_{2,3}^{3}z_{2,4}z_{3,4}z_{4,5}+z_{1,3}^{3}z_{1,4}z_{2,4}z_{2,5}^{3}z_{3,4}z_{4,5}
  \]
  \[
    L_3 =     z_{1,3}^{3}z_{1,5}z_{2,4}^{3}z_{2,5}z_{3,5}z_{4,5}+z_{1,2}^{3}z_{1,5}z_{2,5}z_{3,4}^{3}z_{3,5}z_{4,5}+z_{1,2}z_{1,3}z_{1,4}z_{1,5}z_{2,3}^{3} z_{4,5}^{3}
  \]  
 \[
    L_4 =      z_{1,2}^{3}z_{1,3}z_{2,3}z_{3,4}z_{3,5}z_{4,5}^{3}.
\]
\end{example}

We have the following theorem.

\begin{theorem}  \label{thm!princharacteristic3}
   Assume $(a,b,c) \in P_3$ is fixed. We set 
   \[
            H_0 =  \sum_{t=1}^5  Q(\eta_t) + U_{(a,b,c)}.
    \]
 We consider the minimal graded free resolution 
     \[
             F_\bullet:  \quad  0 \to R(-5)  \to R(-3)^{5} 
                            \to  R(-2)^5  \to  R
     \]
           of  $R/I$   as  $R$-module, and its third  Frobenius power 
\[
        \varphi^*F_\bullet:  \quad  0 \to R(-15)  \to R(-9)^5 
                            \to  R(-6)^5  \to  R
     \]   
 Then there exists   a chain map $\Phi^{\flat}_\bullet\colon \varphi^*F_\bullet  \to F_\bullet$ such that $\Phi^{\flat}_0$ is the identity map and $\Phi^{\flat}_3\colon\varphi^*F_3\cong R(-15)\to R(-5)\cong F_3$ 
           is multiplication by $H_0$.
\end{theorem}  

\begin{proof} 
The proof was given by Macaulay2. The code is contained in the file  \cite{AKOPP53}.
\end{proof}

\subsection{Original Tom}   

The present section contains work related to the codimension $4$ Gorenstein 
ideal $I$ called 
"Original Tom".  Geometrically, the ideal $I$ is the homogeneous ideal of the 
Segre embedding of   $ \mathbb{P}^2 \times \mathbb{P}^2 $
inside $\mathbb{P}^8$, compare \cite{Re} and  \cite[p.~260, Section 5.5]{P2}.
Algebraically, the ideal $I$ is generated  by the $2 \times 2$ minors 
of a generic $3 \times 3$ matrix.

   Assume $\mathbbm{k}$  is a characteristic $2$ field. 
    We consider  the polynomial ring 
    \[
           R \coloneqq  \mathbbm{k} [ z_{i,j} :  1 \leq i , j \leq 3]
    \]
    with the standard grading. Let $G \coloneqq  \mathfrak{S}_3 \times \mathfrak{S}_3$ act on $R$ by  $(\sigma, \tau) \cdot (z_{i,j}) =     z_{\sigma(i),\tau(j)}$.

We consider the $3 \times 3$  matrix $A$ with  $(i,j)$-entry for
equal to  $z_{i,j}$  for all $1 \leq i,j \leq 3$.  We denote by 
$I \subset R$ the ideal generated by all $ 2 \times 2$ minors of $A$.
The ideal $I$ is $G$-stable, in the sense that if
$g \in G$ and $u \in I$ then $g \cdot u \in I$.  It is well-known,
see \cite[p.~16, Section~2.C]{BV}, that
$I$ is a codimension $4$ ideal whose resolution is given by
the Eagon-Northcott complex. We denote by $I^{[2]}$ the second
Frobenius power of $I$.

We consider the polynomials
\[
    \eta_1 = z_{1, 2}z_{1, 3}z_{2, 2}z_{2, 3}z_{3, 1}^2, \quad 
    \eta_2 =   z_{1, 2}z_{1, 3}z_{2, 1}z_{2,3}z_{3, 1}z_{3, 2}, \quad 
    \eta_3       =       z_{1, 2}^2z_{2, 3}^2z_{3, 1}^2.
\]

For $1 \leq i \leq 3$ we denote by $O_i$ the $G$-orbit of the element $\eta_i$.
We have $ \# O_1 = 9$,  $ \# O_2 = 6$,  $ \# O_3 = 6$. Moreover,
\[
   O_3 = \{\eta_3,   z_{1, 3}^2z_{2, 1}^2z_{3, 2}^2, z_{1, 1}^2z_{2, 3}^2z_{3, 2}^2, z_{1, 3}^2z_{2, 2}^2z_{3, 1}^2,  z_{1, 2}^2z_{2, 1}^2z_{3, 3}^2 ,  z_{1, 1}^2z_{2, 2}^2z_{3, 3}^2 \}.
\]

For a subset $S \subset O_3$ of  odd cardinality, we set
\[
   H_{0,S} = \sum_{u \in O_1} u + \sum_{u \in O_2} u  + \sum_{u \in S} u.
\]

\begin{prop} Assume $S \subset O_3$ is a subset of odd cardinality.
    We consider the minimal graded free resolution 
     \[
             F_\bullet:  \quad  0 \to R(-6)  \to R(-4)^9 \to R(-3)^{16} 
                            \to  R(-2)^9  \to  R
     \]
           of  $R/I$   as  $R$-module, and the second Frobenius power 
\[
        \varphi^*F_\bullet:  \quad  0 \to R(-12)  \to R(-8)^9 \to R(-6)^{16} 
                            \to  R(-4)^9  \to  R.
     \]   
     Then there exists 
           a chain map $\Phi_\bullet^{\flat}\colon\varphi^*F_\bullet\to F_\bullet$ such $\Phi_0^\bullet$ is the identity map and that 
           \[\Phi_4^{\flat}\colon \varphi^*F_4\cong R(-12)\to R(-6)\cong F_4\]  
           is multiplication by $H_{0,S}$.
\end{prop}

\begin{proof}  
The proof was given by Macaulay2. The code is contained in the file \cite{AKOPP54}.
\end{proof}

\subsection{Original Jerry}   

We study, the  codimension $4$ Gorenstein 
ideal $I$ called 
"Original Jerry", which is the  the homogeneous ideal of the 
Segre embedding of   $ \mathbb{P}^1 \times \mathbb{P}^1 \times  \mathbb{P}^1$
inside $\mathbb{P}^7$, compare \cite{Re} and  \cite[p.~265]{P2}.
We use the  notation for the variables introduced in  \cite[p.~25, Section~6]{Re}.

   Assume $\mathbbm{k}$  is a characteristic $2$ field. 
    We consider  the polynomial ring 
    \[
           R =  \mathbbm{k} [x,t, y_1,y_2 ,y_3, z_1, z_2 ,z_3]
    \]
    with the standard grading. 
	
 We consider the group
$G=  \mathfrak{S}_3 \times \mathbb{Z}/(2)$ acting on $R$ by 
\[ 
     (\sigma ,[m]) \cdot y_i = y_{\sigma(i)},  \quad  (\sigma ,[m]) \cdot z_i = z_{\sigma(i)}, 
      \quad     (\sigma , [0]) \cdot x = x,  \quad  (\sigma ,[0]) \cdot t = t, 
\]
\[
    (\sigma , [1]) \cdot x = t, \quad   (\sigma ,[1]) \cdot t = x. 
\]

We denote by $I \subset R$ the following ideal 
\[
   I \coloneqq ( xz_1-y_2y_3,   xz_2-y_1y_3,    xz_3-y_1y_2,  
 ty_1 - z_2z_3,   ty_2-z_1z_3,    ty_3-z_1z_2,  xt-y_1z_1,  xt - y_2z_2,   xt - y_3z_3 ).
\]
As already mentioned,   the ideal $I$ is the homogeneous ideal of the 
Segre embedding of   $ \mathbb{P}^1 \times \mathbb{P}^1 \times  \mathbb{P}^1$
inside $\mathbb{P}^7$.  We denote by $I^{[2]}$ the second
Frobenius power of $I$.

We consider the polynomials
\[
   \eta_1=y_1y_2y_3z_1z_2z_3, \phantom{=} \eta_2 = xy_1y_2z_1z_2t, \phantom{=}
   \eta_3 =  xy_1z_1^2z_2z_3,  \phantom{=}   \eta_4 = x^2z_1z_2z_3t, \phantom{=}
  \eta_5 =  x^2y_1z_1t^2    
\]
and the set  
\[
        Q = \{ xy_i^2z_i^2t  :  1 \leq i \leq 3 \}.
\]
The set $Q$ is the $G$-orbit of $xy_1^2z_1^2t$.

For $1 \leq i \leq 5$ we denote by $O_i$ the $G$-orbit of the element $\eta_i$.
We have 
\[
   \# O_1 =  1,  \quad  \# O_2 =  3,  \quad  \# O_3 =  6,  \quad  \# O_4 =  2,
        \quad  \# O_5 =  3.
\]
 For a subset $S \subset Q$ of  odd cardinality, we set
\[
   H_{0,S} = \sum_{u \in S} u +  \sum_{i=1}^{5} \sum_{u \in O_i} u.
\]

\begin{prop} Assume $S \subset Q$ is a subset of odd cardinality.
    We consider the minimal graded free resolution 
     \[
             F_\bullet:  \quad  0 \to R(-6)  \to R(-4)^9 \to R(-3)^{16} 
                            \to  R(-2)^9  \to  R
     \]
           of  $R/I$   as  $R$-module, and the second Frobenius power 
\[
        \varphi^*F_\bullet:  \quad  0 \to R(-12)  \to R(-8)^9 \to R(-6)^{16} 
                            \to  R(-4)^9  \to  R
     \]   
    Then there exists 
           a chain map $\Phi^\flat_\bullet\colon \varphi^*F_\bullet\to F_\bullet$ such that $\Phi^\flat_0$ is the identity map and 
           \[\Phi^{\flat}_4\colon \varphi^*F_4\cong R(-12)\to R(-6)\cong F_4\]          is multiplication by $H_{0,S}$.
\end{prop}

\begin{proof}  

The proof was given by Macaulay2. The code is contained in the file  \cite{AKOPP55}.
\end{proof}

\bibliographystyle{myamsalpha}
\bibliography{references}

@misc{AHL,
	Author = {Karim Adiprasito},
	Eprint = {arXiv:1812.10454},
	Title = {{Combinatorial Lefschetz theorems beyond positivity}},
	Year = {2018},
	note = {\href{https://arxiv.org/abs/1812.10454}{arXiv:1812.10454}},
}

@book {BrunsHerzog,
    AUTHOR = {Bruns, Winfried and Herzog, J\"urgen},
     TITLE = {Cohen-{M}acaulay rings},
    SERIES = {Cambridge Studies in Advanced Mathematics},
    VOLUME = {39},
 PUBLISHER = {Cambridge University Press, Cambridge},
      YEAR = {1993},
     PAGES = {xii+403},
      ISBN = {0-521-41068-1},
   MRCLASS = {13H10 (13-02)},
  MRNUMBER = {1251956},
MRREVIEWER = {Matthew\ Miller},
}

@misc{APPKM,
	Author = {Karim Adiprasito and Stavros Argyrios Papadakis and Vasiliki Petrotou},
	Title = {The volume intrinsic to a commutative graded algebra},
	Year = {2024},
	Eprint = {arXiv:2407.11916},
	note = {\href{https://hal.science/hal-04661500}{hal--04661500}}
}

@misc{PP,
      title={The characteristic 2 anisotropicity of simplicial spheres}, 
      author={Stavros Argyrios Papadakis and Vasiliki Petrotou},
      year={2020},
      eprint={2012.09815},
      archivePrefix={arXiv},
      primaryClass={math.AC},
      url={https://arxiv.org/abs/2012.09815}, 
      note={\href{https://arxiv.org/abs/2012.09815}{arxiv:2012.09815}}
}

@misc{KLS,
      title={Differential operators, anisotropy, and simplicial spheres}, 
      author={Kalle Karu and Matt Larson and Alan Stapledon},
      year={2025},
      eprint={2412.04561},
      archivePrefix={arXiv},
      primaryClass={math.CO},
      url={https://arxiv.org/abs/2412.04561}, 
      note={\href{https://arxiv.org/abs/2412.04561}{arxiv:2412.04561}}
}

@misc{APPlattice,
      title={Lattice polytopes and semigroup algebras: generic {L}efschetz properties and {P}arseval-{R}ayleigh identities}, 
      author={Karim Adiprasito and Stavros Argyrios Papadakis and Vasiliki Petrotou},
      year={2025},
      eprint={2509.14152},
      archivePrefix={arXiv},
      primaryClass={math.CO},
      url={https://arxiv.org/abs/2509.14152}, 
      note={\href{https://hal.science/hal-05238822}{hal--05238822}}
}

@misc{AOPP:complete,
      title={Parseval--{R}ayleigh identities for homogeneous complete intersections}, 
      author={Karim Adiprasito and Ryoshun Oba and Stavros Argyrios Papadakis and Vasiliki Petrotou},
      year={2025},
      eprint={2511.05288},
      archivePrefix={arXiv},
      primaryClass={math.AC},
      url={https://arxiv.org/abs/2511.05288}, 
      note={\href{https://hal.science/hal-05354563}{hal--05354563}}
}

@article{MZ,
	author = {Migliore, Juan and Zanello, Fabrizio},
	title = {The strength of the weak {Lefschetz} property},
	fjournal = {Illinois Journal of Mathematics},
	journal = {Ill. J. Math.},
	issn = {0019-2082},
	volume = {52},
	number = {4},
	pages = {1417--1433},
	year = {2008},
	language = {English},
	zbMATH = {5660666},
	Zbl = {1178.13011}
}

@Misc{M2,
	author = {Grayson, Daniel R. and Stillman, Michael E.},
	title = {Macaulay2, a software system for research in algebraic geometry},
	howpublished = {Available at \url{http://www2.macaulay2.com}},
}

@article {MN:tour,
    AUTHOR = {Migliore, Juan and Nagel, Uwe},
     TITLE = {Survey article: a tour of the weak and strong {L}efschetz
              properties},
   JOURNAL = {J. Commut. Algebra},
  FJOURNAL = {Journal of Commutative Algebra},
    VOLUME = {5},
      YEAR = {2013},
    NUMBER = {3},
     PAGES = {329--358},
      ISSN = {1939-0807,1939-2346},
   MRCLASS = {13E10 (13A02 13C40 13D40)},
  MRNUMBER = {3161738},
MRREVIEWER = {Adela\ N.\ Vraciu},
       DOI = {10.1216/JCA-2013-5-3-329},
       URL = {https://doi-org.proxy.lib.ohio-state.edu/10.1216/JCA-2013-5-3-329},
}

@article{ABD,
	title = {{Lefschetz properties of some codimension three Artinian Gorenstein algebras}},
	journal = {Journal of Algebra},
	volume = {625},
	pages = {28-45},
	year = {2023},
	issn = {0021-8693},
	doi = {https://doi.org/10.1016/j.jalgebra.2023.03.005},
	url = {https://www.sciencedirect.com/science/article/pii/S0021869323001059},
	author = {Nancy Abdallah and Nasrin Altafi and Anthony Iarrobino and Alexandra Seceleanu and Joachim Yaméogo}
}

@article {Batyrev:variations,
    AUTHOR = {Batyrev, Victor V.},
     TITLE = {Variations of the mixed {H}odge structure of affine
              hypersurfaces in algebraic tori},
   JOURNAL = {Duke Math. J.},
  FJOURNAL = {Duke Mathematical Journal},
    VOLUME = {69},
      YEAR = {1993},
    NUMBER = {2},
     PAGES = {349--409},
      ISSN = {0012-7094,1547-7398},
   MRCLASS = {14M25 (14D07 14F40 14J45 32J25)},
  MRNUMBER = {1203231},
MRREVIEWER = {Richard\ M.\ Hain},
       DOI = {10.1215/S0012-7094-93-06917-7},
       URL = {https://doi.org/10.1215/S0012-7094-93-06917-7},
}

@incollection {Braun:unimodality,
    AUTHOR = {Braun, Benjamin},
     TITLE = {Unimodality problems in {E}hrhart theory},
 BOOKTITLE = {Recent trends in combinatorics},
    SERIES = {IMA Vol. Math. Appl.},
    VOLUME = {159},
     PAGES = {687--711},
 PUBLISHER = {Springer, [Cham]},
      YEAR = {2016},
      ISBN = {978-3-319-24296-5; 978-3-319-24298-9},
   MRCLASS = {52B20 (05A20)},
  MRNUMBER = {3526428},
MRREVIEWER = {Ruriko\ Yoshida},
       DOI = {10.1007/978-3-319-24298-9\_27},
       URL = {https://doi-org.proxy.lib.ohio-state.edu/10.1007/978-3-319-24298-9_27},
}

@article {DimcaIlardi,
    AUTHOR = {Dimca, Alexandru and Ilardi, Giovanna},
     TITLE = {Lefschetz properties of {J}acobian algebras and {J}acobian
              modules},
   JOURNAL = {Ann. Sc. Norm. Super. Pisa Cl. Sci. (5)},
  FJOURNAL = {Annali della Scuola Normale Superiore di Pisa. Classe di
              Scienze. Serie V},
    VOLUME = {26},
      YEAR = {2025},
    NUMBER = {3},
     PAGES = {1603--1616},
      ISSN = {0391-173X,2036-2145},
   MRCLASS = {32S25 (13A02 13D02 13E10 14B05)},
  MRNUMBER = {4991580},
}

@article {Stanley:Hilbertfunctions,
    AUTHOR = {Stanley, Richard P.},
     TITLE = {Hilbert functions of graded algebras},
   JOURNAL = {Advances in Math.},
  FJOURNAL = {Advances in Mathematics},
    VOLUME = {28},
      YEAR = {1978},
    NUMBER = {1},
     PAGES = {57--83},
      ISSN = {0001-8708},
   MRCLASS = {13D10 (13H10)},
  MRNUMBER = {485835},
MRREVIEWER = {Idun\ Reiten},
       DOI = {10.1016/0001-8708(78)90045-2},
       URL = {https://doi-org.proxy.lib.ohio-state.edu/10.1016/0001-8708(78)90045-2},
}

@incollection {Watanabe,
    AUTHOR = {Watanabe, Junzo},
     TITLE = {The {D}ilworth number of {A}rtinian rings and finite posets
              with rank function},
 BOOKTITLE = {Commutative algebra and combinatorics ({K}yoto, 1985)},
    SERIES = {Adv. Stud. Pure Math.},
    VOLUME = {11},
     PAGES = {303--312},
 PUBLISHER = {North-Holland, Amsterdam},
      YEAR = {1987},
      ISBN = {0-444-70314-4},
   MRCLASS = {13E10 (13C99)},
  MRNUMBER = {951211},
MRREVIEWER = {W.\ Heinzer},
       DOI = {10.2969/aspm/01110303},
       URL = {https://doi-org.proxy.lib.ohio-state.edu/10.2969/aspm/01110303},
}

@article {THamiltonian,
    AUTHOR = {Thomason, A. G.},
     TITLE = {Hamiltonian cycles and uniquely edge colourable graphs},
   JOURNAL = {Ann. Discrete Math.},
  FJOURNAL = {Annals of Discrete Mathematics},
    VOLUME = {3},
      YEAR = {1978},
     PAGES = {259--268},
   MRCLASS = {05C45 (05C15)},
  MRNUMBER = {499124},
MRREVIEWER = {J.\ Bos\'ak},
}

@book{BV,
  author    = {Winfried Bruns and Udo Vetter},
  title     = {Determinantal Rings},
  series    = {Lecture Notes in Mathematics},
  volume    = {1327},
  publisher = {Springer-Verlag},
  address   = {Berlin},
  year      = {1988}
}

@misc{GP,
  author = {{Groupprops}},
  title  = {Subgroup structure of symmetric group: S5},
  note   = {\url{https://groupprops.subwiki.org/wiki/Subgroup_structure_of_symmetric_group:S5}}
}

@misc{GP2,
  author = {{Groupprops}},
  title  = {General affine group: GA(1,5)},
  note   = {\url{https://groupprops.subwiki.org/wiki/General_affine_group:GA(1,5)}}
}

@incollection {Re,
    AUTHOR = {Miles Reid},
     TITLE = {Graded Rings and Birational Geometry},
 BOOKTITLE = {Proc. of algebraic symposium (Kinosaki, Oct 2000)},
     PAGES = {1--72}
}

@article{P2,
  author  = {Stavros A. Papadakis},
  title   = {Kustin--{M}iller unprojection with complexes},
  journal = {Journal of Algebraic Geometry},
  volume  = {13},
  year    = {2004},
  pages   = {249--268}
}

@misc{Po,
      title={Parseval-{R}ayleigh identities for graded {A}rtinian {G}orenstein algebras}, 
      author={Mykola Pochekai},
      year={2026},
      eprint={2604.27631},
      archivePrefix={arXiv},
      primaryClass={math.AC},
      url={https://arxiv.org/abs/2604.27631},
      note={\href{https://arxiv.org/abs/2604.27631}{arxiv:2604.27631}} 
}

@misc{AKOPP51,
   author = {Karim Adiprasito and Eric Katz and Ryoshun Oba and Stavros
Argyrios Papadakis and Vasiliki Petrotou},
   title = {Macaulay2 Code 51},
   year = {2026},
   note = {HAL Id: hal-05609907}
}

@misc{AKOPP52,
   author = {Karim Adiprasito and Eric Katz and Ryoshun Oba and Stavros
Argyrios Papadakis and Vasiliki Petrotou},
   title = {Macaulay2 Code 52},
   year = {2026},
   note = {HAL Id: hal-05609914}
}

@misc{AKOPP53,
   author = {Karim Adiprasito and Eric Katz and Ryoshun Oba and
Stavros Argyrios Papadakis and Vasiliki Petrotou},
   title = {Macaulay2 Code 53},
   year = {2026},
   note = {HAL Id: hal-05609918}
}

@misc{AKOPP54,
   author = {Karim Adiprasito and Eric Katz and Ryoshun Oba and
Stavros Argyrios Papadakis and Vasiliki Petrotou},
   title = {Macaulay2 Code 54},
   year = {2026},
   note = {HAL Id: hal-05609923}
}

@misc{AKOPP55,
   author = {Karim Adiprasito and Eric Katz and Ryoshun Oba and
Stavros Argyrios Papadakis and Vasiliki Petrotou},
   title = {Macaulay2 Code 55},
   year = {2026},
   note = {HAL Id: hal-05609927}
}

@unpublished{APPhal,
	TITLE = {{Combinatorial Lefschetz theorems beyond positivity II: Total Anisotropy and its variants}},
	AUTHOR = {Karim Adiprasito and Stavros Argyrios Papadakis and Petrotou, Vasiliki},
         URL = {https://hal.science/hal-05124466},
	YEAR = {2023},
	MONTH = Jun,
      note={\href{https://hal.science/hal-05124466}{hal--05124466}},
	HAL_ID = {hal--05124466},
	HAL_VERSION = {v1},
}

\end{document}